\documentclass[10pt]{article}
\usepackage{palatino}
\usepackage{amsmath}
\usepackage{amsthm}
\usepackage{amssymb}
\usepackage{url}
\usepackage{latexsym}
\usepackage[xcolor=pst]{pstricks}
\usepackage{graphicx, pst-plot, pst-node, pst-text, pst-tree}
\usepackage{titlefoot}
\usepackage[small]{titlesec}
\usepackage{units} 
\usepackage[small,it]{caption}
\usepackage{arydshln}
\usepackage{enumitem}
\usepackage{mathabx}

\usepackage{color}
\definecolor{lightgray}{rgb}{0.8, 0.8, 0.8}
\definecolor{darkgray}{rgb}{0.7, 0.7, 0.7}
\definecolor{darkblue}{rgb}{0, 0, .4}

\usepackage[bookmarks]{hyperref}
\hypersetup{
        colorlinks=true,
        linkcolor=darkblue,
        anchorcolor=darkblue,
        citecolor=darkblue,
        urlcolor=darkblue,
        pdfpagemode=UseThumbs,
        pdftitle={Inflations of geometric grid classes of permutations},
        pdfsubject={Combinatorics},
        pdfauthor={Albert, Ruskuc, and Vatter},
}

\newcounter{todocounter}

\theoremstyle{plain}
\newtheorem{theorem}{Theorem}[section]
\newtheorem{proposition}[theorem]{Proposition}

\newtheorem{corollary}[theorem]{Corollary}

\theoremstyle{definition}

\newtheorem{conjecture}[theorem]{Conjecture}
\newtheorem{question}[theorem]{Question}

\makeatletter
\newfont{\footsc}{cmcsc10 at 8truept}
\newfont{\footbf}{cmbx10 at 8truept}
\newfont{\footrm}{cmr10 at 10truept}
\renewenvironment{abstract}%
                {
                  \begin{list}{}%
                     {\setlength{\rightmargin}{1in}%
                      \setlength{\leftmargin}{1in}}%
                   \item[]\ignorespaces\begin{small}}%
                 {\end{small}\unskip\end{list}}

\newcommand{\Av}{\operatorname{Av}}

\newcommand{\C}{\mathcal{C}}
\newcommand{\D}{\mathcal{D}}

\newcommand{\M}{\mathcal{M}}
\renewcommand{\P}{\mathcal{P}}
\newcommand{\R}{\mathcal{R}}
\renewcommand{\S}{\mathcal{S}}
\newcommand{\T}{\mathcal{T}}
\newcommand{\U}{\mathcal{U}}
\newcommand{\V}{\mathcal{V}}
\renewcommand{\O}{\mathcal{O}}
\newcommand{\Q}{\mathcal{Q}}

\newcommand{\X}{\mathcal{X}}
\newcommand{\Y}{\mathcal{Y}}
\newcommand{\Z}{\mathcal{Z}}

\newcommand{\fF}{\mathfrak{F}}

\newcommand{\gr}{\mathrm{gr}}
\newcommand{\lgr}{\underline{\gr}}
\newcommand{\ugr}{\overline{\gr}}
\newcommand{\zpm}{0/\mathord{\pm} 1}
\newcommand{\cell}[2]{a_{#1#2}}
\newcommand{\interval}[2]{[#1, #2]}

\newcommand{\Grid}{\operatorname{Grid}}
\newcommand{\Geom}{\operatorname{Geom}}

\newcommand{\st}{\::\:}

\newcommand{\bij}{\varphi}

\newcommand{\initcomp}{{\rm \#1}}

\newcommand{\emptyword}{\varepsilon}
\newcommand{\oplusprop}{D_{\mathord{\oplus}}}
\newcommand{\ominusprop}{D_{\mathord{\ominus}}}
\newcommand{\suplessthan}{\mbox{\begin{tiny}\ensuremath{<}\end{tiny}}}
\newcommand{\fnmatrix}[2]{\mbox{\begin{footnotesize}$\left(\begin{array}{#1}#2\end{array}\right)$\end{footnotesize}}}

\newrgbcolor{gray69}{0.7 0.7 0.7}
\newrgbcolor{gray53}{0.53 0.53 0.53}
\newrgbcolor{gray80}{0.8 0.8 0.8}
\newrgbcolor{gray90}{0.90 0.90 0.90}
\newrgbcolor{gray95}{0.95 0.95 0.95}

\datefoot{\today}
\amssubj{Primary: 05A15; Secondary: 68Q45, 05A05}
\keywords{algebraic generating function, context-free language, permutation class, regular language, restricted permutation, simple permutation, substitution decomposition}

\newpagestyle{main}[\small]{
        \headrule
        \sethead[\usepage][][]
        {\sc Inflations of Geometric Grid Classes of Permutations}{}{\usepage}}

\title{\sc Inflations of Geometric Grid Classes of Permutations}
\author{%
Michael H. Albert\footnote{All three authors were partially supported by EPSRC via the grant EP/J006440/1.}\\[-0.25ex]
\small Department of Computer Science\\[-0.5ex]
\small University of Otago\\[-0.5ex]
\small Dunedin, New Zealand\\[1.5ex]
Nik Ru\v{s}kuc\footnotemark[\value{footnote}]\\[-0.25ex]
\small School of Mathematics and Statistics\\[-0.5ex]
\small University of St Andrews\\[-0.5ex]
\small St Andrews, Scotland\\[1.5ex]
Vincent Vatter\footnotemark[\value{footnote}]\footnote{The third author was also partially supported by the NSA Young Investigator Grant 12-1-0207.}\\[-0.25ex]
\small Department of Mathematics\\[-0.5ex]
\small University of Florida\\[-0.5ex]
\small Gainesville, Florida USA\\[-1.5ex]
}

\titleformat{\section}
        {\large\sc}
        {\thesection.}{1em}{}   

\date{}

\begin{document}
\maketitle

\pagestyle{main}

\begin{abstract}
Geometric grid classes and the substitution decomposition have both been shown to be fundamental in the understanding of the structure of permutation classes.  In particular, these are the two main tools in the recent classification of permutation classes of growth rate less than $\kappa\approx2.20557$ (a specific algebraic integer at which infinite antichains begin to appear).  Using language- and order-theoretic methods, we prove that the substitution closures of geometric grid classes are partially well-ordered, finitely based, and that all their subclasses
have algebraic generating functions.  
We go on to show that the inflation of a geometric grid class by a strongly rational class is partially well-ordered, and that
all its subclasses have rational generating functions.  This latter fact allows us to conclude that every permutation class with growth rate less than $\kappa$ has a rational generating function.  This bound is tight as there are permutation classes with growth rate $\kappa$ which have nonrational generating functions.
\end{abstract}

\maketitle

\section{Introduction}\label{infinite-simples-intro}

The celebrated proof of the Stanley--Wilf Conjecture by Marcus and Tardos~\cite{marcus:excluded-permut:} establishes that all nontrivial permutation classes have at most exponential growth.  A prominent line of subsequent research  has focused on determining the possible growth rates of these classes.  In particular, Vatter~\cite{vatter:small-permutati:} characterised all growth rates up to
\[
\kappa=\mbox{the unique real root of $x^3-2x^2-1$}\approx 2.20557.
\]
The number $\kappa$ is the threshold of a sharp phase transition: there are only countably many permutation classes of growth rate less than $\kappa$, but uncountably many of growth rate $\kappa$. 
Furthermore, it is the first growth rate at which permutation classes may contain infinite antichains,
which in turn is the cause of much more complicated structure.  
For this reason we single out classes of growth rate less than $\kappa$ as \emph{small}. 
In this work we elucidate the enumerative structure of small permutation classes, essentially completing this research programme by proving that \emph{all small permutation classes have rational generating functions}.


Our work combines and extends two of the most useful techniques for analysing the structure of permutation classes: geometric grid classes and the substitution decomposition.  The conclusion about small permutation classes is obtained from a general enumerative result showing that the inflation of a geometric grid class by a strongly rational class is itself strongly rational.  This also places Theorem 3.5 of Albert, Atkinson, and Vatter~\cite{albert:subclasses-of-t:} in a wider theoretical context.  We introduce our results gradually, and along the way prove another structural result of independent interest: the substitution closure of a geometric grid class is partially well-ordered, finitely based, and all its subclasses have algebraic generating functions.  Thereby, we generalise one of the main theorems of Albert and Atkinson~\cite{albert:simple-permutat:} to a more natural and applicable setting.

For the rest of the introduction, we give just enough notation to motivate and precisely state our main results.  Sections~\ref{sec-inflate} and \ref{sec-geomgrid} offer a more thorough review of the substitution decomposition and geometric grid classes.  Our new results are proved in Sections~\ref{sec-infinite-finite-bases-pwo}--\ref{sec-spc-rational}, and Section~\ref{sec-conclusion} concludes by outlining several directions for further investigation.

Given permutations $\pi$ and $\sigma$, we say that $\pi$ \emph{contains} $\sigma$, and write $\sigma\le\pi$, if $\pi$ has a subsequence $\pi(i_1)\cdots\pi(i_k)$ of length $k$ which is order isomorphic to $\sigma$; otherwise, we say that $\pi$ \emph{avoids} $\sigma$.  For example, $\pi=391867452$ (written in list, or one-line notation) contains $\sigma=51342$, as can be seen by considering the subsequence $\pi(2)\pi(3)\pi(5)\pi(6)\pi(9)=91672$.  A \emph{permutation class} is a downset, say $\C$, of permutations under this order; i.e., if $\pi\in\C$ and $\sigma\le\pi$, then $\sigma\in\C$.  

For any permutation class $\C$ there is a unique, possibly infinite, antichain $B$ such that
\[
\C=\Av(B)=\{\pi: \pi \not \geq\beta\mbox{ for all } \beta \in B\}.
\]
This antichain $B$, which consists of all the minimal permutations \emph{not} in $\C$, is called the \emph{basis} of $\C$. 
If $B$ happens to be finite, we say that $\C$ is \emph{finitely based}.
 For $n\in\mathbb{N}$, we denote by $\C_n$ the set of permutations in $\C$ of length $n$, and we refer to
\[
\sum_{n=0}^\infty |\C_n|x^n=\sum_{\pi\in\C} x^{|\pi|}
\]
as the \emph{generating function} of $\C$ (here $|\pi|$ denotes the length of the permutation $\pi$).  Since proper permutation classes are of exponentially bounded size, they have associated parameters of interest related to their asymptotic growth. Specifically, every class $\C$ has \emph{upper} and \emph{lower growth rates} given, respectively, by
\[
\ugr(\C)=\limsup_{n\rightarrow\infty}\sqrt[n]{|\C_n|}
\quad\mbox{and}\quad
\lgr(\C)=\liminf_{n\rightarrow\infty}\sqrt[n]{|\C_n|}.
\]
It is conjectured that the actual limit of $\sqrt[n]{|\C_n|}$ exists for every permutation class; whenever this limit is known to exist we call it the \emph{(proper) growth rate} of $\C$ and denote it by $\gr(\C)$.  (All growth rates mentioned in the first paragraph are proper growth rates.)

A partially ordered set (poset for short) is said to be \emph{partially well-ordered (pwo)} if it contains neither an infinite strictly descending sequence nor an infinite antichain.  In permutation classes, pwo is synonymous with the absence of infinite antichains, since they cannot contain infinite strictly decreasing sequences .

Geometric grid classes are the first of our major tools, and may be defined as follows.  Suppose that $M$ is a $\zpm$ matrix.  The \emph{standard figure} of $M$ is the point set in $\mathbb{R}^2$ consisting of:
\begin{itemize}
\item the line segment from $(k-1,\ell-1)$ to $(k,\ell)$ if $M_{k,\ell}=1$ or
\item the line segment from $(k-1,\ell)$ to $(k,\ell-1)$ if $M_{k,\ell}=-1$.
\end{itemize}
The \emph{geometric grid class} of $M$, denoted by $\Geom(M)$, is then the set of all permutations that can be drawn on this figure in the following manner.  Choose $n$ points in the figure, no two on a common horizontal or vertical line.  Then label the points from $1$ to $n$ from bottom to top and record these labels reading left to right.  An example is shown in Figure~\ref{fig-example-ggc}.  Note that in order for the cells of the matrix $M$ to be compatible with plots of permutations, we use Cartesian coordinates for our matrices, indexing them first by column, from left to right starting with $1$, and then by row, from bottom to top.

A permutation class is said to be \emph{geometrically griddable} if it is contained in a geometric grid class.  These classes are known to be well behaved:

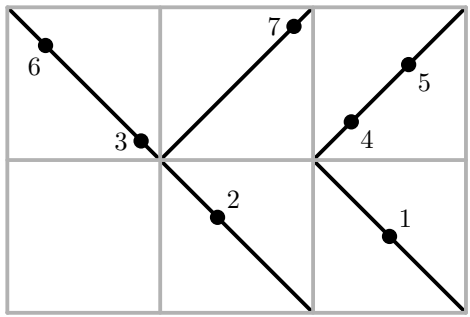
\begin{figure}
\begin{center}
\psset{xunit=0.02in, yunit=0.02in}
\psset{linewidth=0.005in}
\begin{pspicture}(-3,-3)(120,80)
\psline[linecolor=black,linestyle=solid,linewidth=0.02in](0,80)(40,40)
\psline[linecolor=black,linestyle=solid,linewidth=0.02in](40,40)(80,0)
\psline[linecolor=black,linestyle=solid,linewidth=0.02in](40,40)(80,80)
\psline[linecolor=black,linestyle=solid,linewidth=0.02in](80,40)(120,0)
\psline[linecolor=black,linestyle=solid,linewidth=0.02in](80,40)(120,80)
\psline[linecolor=darkgray,linestyle=solid,linewidth=0.02in]{c-c}(0,0)(0,80)
\psline[linecolor=darkgray,linestyle=solid,linewidth=0.02in]{c-c}(40,0)(40,80)
\psline[linecolor=darkgray,linestyle=solid,linewidth=0.02in]{c-c}(80,0)(80,80)
\psline[linecolor=darkgray,linestyle=solid,linewidth=0.02in]{c-c}(120,0)(120,80)
\psline[linecolor=darkgray,linestyle=solid,linewidth=0.02in]{c-c}(0,0)(120,0)
\psline[linecolor=darkgray,linestyle=solid,linewidth=0.02in]{c-c}(0,40)(120,40)
\psline[linecolor=darkgray,linestyle=solid,linewidth=0.02in]{c-c}(0,80)(120,80)
\pscircle*(10,70){0.04in} 
\uput[240](10,70){$6$}

\pscircle*(35,45){0.04in} 
\uput[180](35,45){$3$}

\pscircle*(55,25){0.04in} 
\uput[45](55,25){$2$}

\pscircle*(75,75){0.04in} 
\uput[180](75,75){$7$}

\pscircle*(90,50){0.04in} 
\uput[-45](90,50){$4$}

\pscircle*(100,20){0.04in} 
\uput[45](100,20){$1$}

\pscircle*(105,65){0.04in} 
\uput[-45](105,65){$5$}

\end{pspicture}
\end{center}
\caption[An example of a geometric grid class.]{The permutation $6327415$ lies in the geometric grid class of the matrix
$$
\fnmatrix{rrr}{-1&1&1\\0&-1&-1}.
$$}
\label{fig-example-ggc}
\end{figure}

\begin{theorem}[Albert, Atkinson, Bouvel, Ru\v{s}kuc and Vatter~\cite{albert:geometric-grid-:}]
\label{thm-geom-griddable-all-summary}
Every geometrically griddable class $\C$ is finitely based, pwo, and has a rational generating function.
\end{theorem}

Theorem~\ref{thm-geom-griddable-all-summary} implies that every geometrically griddable class $\C$ is \emph{strongly rational}, in the sense that $\C$ and all of its subclasses have rational generating functions.  Strong rationality has numerous consequences for permutation classes, such as the following, which can be established by a simple counting argument.

\begin{proposition}[Albert, Atkinson, and Vatter~\cite{albert:subclasses-of-t:}]
\label{prop-strong-rat-pwo}
Every strongly rational class is pwo.
\end{proposition}

Our second major tool is the \emph{substitution decomposition} of permutations into intervals.  An \emph{interval} in the permutation $\pi$ is a set of contiguous indices $I=\interval{a}{b}=\{a,a+1,\dots,b\}$ such that the set of values $\pi(I)=\{\pi(i) : i\in I\}$ is also contiguous.  Given a permutation $\sigma$ of length $m$ and nonempty permutations $\alpha_1,\dots,\alpha_m$, the \emph{inflation} of $\sigma$ by $\alpha_1,\dots,\alpha_m$ is the permutation $\pi=\sigma[\alpha_1,\dots,\alpha_m]$ obtained by replacing each entry $\sigma(i)$ by an interval that is order isomorphic to $\alpha_i$.  For example,
\[
2413[1,132,321,12]=4\ 798\ 321\ 56. 
\]
Given two classes $\C$ and $\U$, the \emph{inflation} of $\C$ by $\U$ is defined as
\[
\C[\U]=\{\sigma[\alpha_1,\dots,\alpha_m]\st\mbox{$\sigma\in\C_m$ and $\alpha_1,\dots,\alpha_m\in\U$}\}.
\]
The class $\C$ is said to be \emph{substitution closed} if $\C[\C]\subseteq\C$. The \emph{substitution closure} $\langle\C\rangle$ of a class $\C$ is defined as the smallest substitution closed class containing $\C$.  A standard argument shows that $\langle\C\rangle$ exists, and a detailed construction of $\langle\C\rangle$ is provided by Proposition~\ref{prop-subst-completion-const}.

The main results of this paper can now be stated.
\begin{itemize}
\item If the class $\C$ is geometrically griddable, then every subclass of $\langle\C\rangle$ is finitely based and pwo (Theorem~\ref{thm-geom-simples-pwo-basis}).
\item If the class  $\C$ is geometrically griddable, then every subclass of $\langle\C\rangle$ has an algebraic generating function (Theorem~\ref{thm-context-free}).
\item If the class  $\C$ is geometrically griddable and $\U$ is strongly rational, then $\C[\U]$ is also strongly rational (Theorem~\ref{thm-geom-inflate-enum}).
\item Every small permutation class has a rational generating function (Theorem~\ref{thm-small-rational}).
\end{itemize}
As mentioned earlier, there are uncountably many classes with growth rate $\kappa$ which give rise to uncountably many different generating functions.
Hence there are permutation classes of growth rate $\kappa$ with nonrational, nonalgebraic, and even nonholonomic generating functions, so that
the final result above is best possible.

\section{Substitution Closures and Simple Permutations}
\label{sec-inflate}

Every permutation of length $n\ge 1$ has \emph{trivial} intervals of lengths $0$, $1$, and $n$; all other intervals are termed \emph{proper}.  A nontrivial permutation is said to be \emph{simple} if it has no proper intervals.  The shortest simple permutations are thus $12$ and $21$, there are no simple permutations of length three, and the simple permutations of length four are $2413$ and $3142$.  Several examples of simple permutations are plotted throughout the paper, for instance in Figures~\ref{fig-par-alts} and \ref{fig-osc}.

Inflations of $12$ and $21$ generally require special treatment and have their own names.  The \emph{(direct) sum} of the permutations $\alpha_1$ and $\alpha_2$ is $\alpha_1\oplus\alpha_2=12[\alpha_1,\alpha_2]$.  The permutation $\pi$ is said to be \emph{sum decomposable} if it can be expressed as a sum of two nonempty permutations, and \emph{sum indecomposable} otherwise.  
For every permutation $\pi$ there are unique sum indecomposable permutations $\alpha_1,\dots,\alpha_k$ (called the 
\emph{sum components} of $\pi$)  such that $\pi=\alpha_1\oplus\dots\oplus\alpha_k$.
A \emph{sum-prefix} and a \emph{sum-suffix} of $\pi$ are any of the permutations $\alpha_1\oplus\dots\oplus\alpha_i$ and $\alpha_i\oplus\dots\oplus \alpha_k$ for $i=1,\dots,k$.
The \emph{skew sum} operation is defined by $\alpha_1\ominus\alpha_2=21[\alpha_1,\alpha_2]$ and the notions of \emph{skew decomposable}, \emph{skew indecomposable}, \emph{skew components}, \emph{skew-prefix}, and \emph{skew-suffix} are defined analogously.

Simple permutations and inflations are linked by the following result.

\begin{proposition}[Albert and Atkinson~\cite{albert:simple-permutat:}]
\label{simple-decomp-unique}
Every nontrivial permutation $\pi$ is an inflation of a unique simple permutation $\sigma$.  Moreover, if $\pi=\sigma[\alpha_1,\dots,\alpha_m]$ for a simple permutation $\sigma$ of length $m\ge 4$, then each $\alpha_i$ is unique.  If $\pi$ is an inflation of $12$ (i.e., is sum decomposable), then there is a unique sum indecomposable $\alpha_1$  such that $\pi=12[\alpha_1,\alpha_2]$.  The same holds, mutatis mutandis, with $12$ replaced by $21$ and sum replaced by skew.
\end{proposition}

We need several technical details about inflations, and we begin by investigating the intervals of an inflation $\pi = \sigma[\alpha_1, \dots, \alpha_m]$. Trivially, any subinterval of any $\alpha_i$ will be an interval of $\pi$, and if $\tau$ is an interval of $\sigma$ corresponding to indices $i$ through $j$, then the entries corresponding to $\tau[\alpha_i,\dots,\alpha_j]$ will also form an interval of $\pi$. Any other interval contains an interval of this second type, possibly together with some entries of $\alpha_{i-1}$ and $\alpha_{j+1}$.

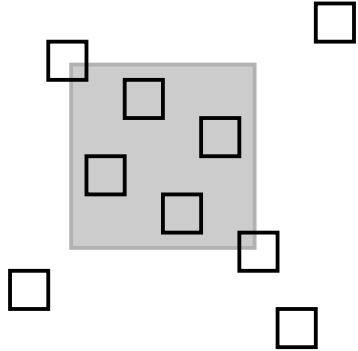
\begin{figure}
\begin{center}
\psset{xunit=0.02in, yunit=0.02in}
\psset{linewidth=0.005in}
\begin{pspicture}(0,0)(100,100)

\pspolygon[linecolor=darkgray,linestyle=solid,linewidth=0.02in,fillstyle=solid,fillcolor=lightgray](26,74)(74,74)(74,26)(26,26)
\pspolygon[linecolor=black,linestyle=solid,linewidth=0.02in](10,10)(20,10)(20,20)(10,20)
\pspolygon[linecolor=black,linestyle=solid,linewidth=0.02in](20,70)(30,70)(30,80)(20,80)
\pspolygon[linecolor=black,linestyle=solid,linewidth=0.02in](30,40)(40,40)(40,50)(30,50)
\pspolygon[linecolor=black,linestyle=solid,linewidth=0.02in](40,60)(50,60)(50,70)(40,70)
\pspolygon[linecolor=black,linestyle=solid,linewidth=0.02in](50,30)(60,30)(60,40)(50,40)
\pspolygon[linecolor=black,linestyle=solid,linewidth=0.02in](60,50)(70,50)(70,60)(60,60)
\pspolygon[linecolor=black,linestyle=solid,linewidth=0.02in](70,20)(80,20)(80,30)(70,30)
\pspolygon[linecolor=black,linestyle=solid,linewidth=0.02in](80,0)(90,0)(90,10)(80,10)
\pspolygon[linecolor=black,linestyle=solid,linewidth=0.02in](90,80)(100,80)(100,90)(90,90)
\end{pspicture}
\end{center}
\caption{An illustration of an interval in an inflation of $285746319$. The small boxes represent the inflations of each point. The large shaded box captures the complete inflations of the interval $5746$ along with (possibly) some skew-suffix of the inflation of the entry $8$, and some skew-prefix of the inflation of the entry $3$.}
\label{fig-example-interval-in-inflation}
\end{figure}

\begin{proposition}
\label{intervals-in-inflation}
Suppose that $\pi = \sigma[\alpha_1,\dots,\alpha_m]$ and that $\theta$ is an interval of $\pi$ not contained in a single $\alpha_i$. Then there exist a (possibly empty) interval $\interval{i}{j}$ of indices, and intervals $\gamma_{i-1}$ and $\gamma_{j+1}$ of $\alpha_{i-1}$ and $\alpha_{j+1}$ respectively, such that $\tau =   \sigma(\interval{i}{j})$ is an interval of $\sigma$, and the entries of $\theta$ correspond to
\[
\gamma_{i-1} \oplus \tau[\alpha_i,\dots,\alpha_j] \oplus \gamma_{j+1}
\quad \mbox{or} \quad
\gamma_{i-1} \ominus \tau[\alpha_i,\dots,\alpha_j] \ominus \gamma_{j+1}.
\]
In the first case $\gamma_{i-1}$ may be nonempty only if $\sigma(\interval{i-1}{j})$ is also an interval of $\sigma$ of the form $1 \oplus \tau$ and $\gamma_{i-1}$ is a sum-suffix of $\alpha_{i-1}$, while $\gamma_{j+1}$ may be nonempty only if $\sigma(\interval{i}{j+1})$ is also an interval of $\sigma$ of the form $\tau \oplus 1$ and $\gamma_{j+1}$ is a sum-prefix of $\alpha_{j+1}$. Analogous conditions apply for the second alternative.
\end{proposition}

\begin{proof}
The set of indices $k$ such that $\alpha_k$ is wholly contained in $\theta$ forms an interval $I = \interval{i}{j}$. Suppose first that $I$ is nonempty.

Observe that if two intervals of a permutation overlap, but neither is contained in the other, then their intersection is a sum- or skew-suffix of one and a sum- or skew-prefix of the other.  So, in the case where $\alpha_{i-1}$ and $\theta$ have entries in common, these entries $\beta_{i-1}$ must lie with respect to $\tau[\alpha_i,\dots,\alpha_j]$ as specified in the statement of the proposition, and similar conclusions apply to $\alpha_{j+1}$. 

In case that $I$ is empty, $\theta$ must have nontrivial intersection with exactly two consecutive intervals $\alpha_{i-1}$ and $\alpha_{j+1}$. Then a minor modification of the argument above (relating $\beta_{i-1}$ directly to $\beta_{j+1}$) completes the proof.
\end{proof}

Inflations $\C[\U]$ of one class by another feature prominently in what follows.  In this area we follow in the footsteps of Brignall~\cite{brignall:wreath-products:}.  Let us define  a \emph{$\U$-inflation} of $\sigma$ to be any permutation of the form
\[
\pi=\sigma[\alpha_1,\dots,\alpha_m]\mbox{ with $\alpha_1,\dots,\alpha_m\in\U$;}
\]
we also refer to the above expression as a \emph{$\U$-decomposition} of $\pi$.

\begin{proposition}[Brignall~\cite{brignall:wreath-products:}]
\label{prop-U-profile}
For every nonempty permutation class $\U$ and every permutation $\pi$, the set
\[
\{\sigma\st\mbox{\textup{$\pi$ can be expressed as a $\U$-inflation of $\sigma$}}\}
\]
has a unique minimal element (with respect to the permutation containment order).
\end{proposition}


This unique minimal $\sigma$ is called the \emph{$\U$-profile} of $\pi$; note that the $\U$-profile of $\pi$ is $1$ precisely if $\pi\in\U$.  To verify that $\pi\in\C[\U]$, it suffices to check whether the $\U$-profile of $\pi$ lies in $\C$.  However, for our enumerative intents, Proposition~\ref{prop-U-profile} is not sufficient, because it does not 
guarantee uniqueness of substitution decomposition, even after the $\U$-profile has been singled out.
For a very simple example, consider the permutation $12345$.  The $\Av(123)$-profile of this permutation is $123$, but it has three  decompositions with respect to this profile:
\[
12345=123[12,12,1]=123[12,1,12]=123[1,12,12].
\]
To address this problem we introduce the \emph{left-greedy} $\U$-decomposition of the permutation $\pi$ as $\sigma[\alpha_1,\dots,\alpha_m]$, where $\sigma$ is the $\U$-profile of $\pi$ and 
the $\alpha_i\in\U$ are chosen from left to right to be as long as possible.
We may also refer to such $\sigma[\alpha_1,\dots,\alpha_m]$ as a \emph{left-greedy} $\U$-inflation of $\sigma$.
By definition the left-greedy $\U$-decomposition is unique, and the question arises as to how to distinguish it from other $\U$-decompositions of $\pi$.
Our next proposition shows that if a $\U$-decomposition is not left-greedy then either several of the $\alpha_i$ can be merged, or a sum- or skew-prefix of one $\alpha_i$ can be appended to $\alpha_{i-1}$.

\begin{proposition}
\label{prop-left-greedy-U-inflate}
Let $\U$ be a nonempty permutation class. A $\U$-decomposition
\[
\pi=\theta[\gamma_1, \dots, \gamma_k]\mbox{ where $\gamma_1,\dots,\gamma_k\in\U$}
\]
is \emph{not} the left-greedy $\U$-decomposition of $\pi$ if and only if there is an interval $\interval{i}{j}$ of length at least $2$ giving rise to an interval of $\theta$ which is order isomorphic to a permutation $\tau$ such that
\begin{enumerate}
\item[(G1)] $\tau[\gamma_i,\gamma_{i+1},\dots,\gamma_j]\in\U$ (implying that $\theta$ is in fact not the $\U$-profile of $\pi$); or
\item[(G2)] $\tau=12$ and the sum of $\gamma_i$ with the first sum component  of $\gamma_{i+1}$ lies in $\U$; or, similarly,
\item[(G3)] $\tau=21$ and the skew sum of $\gamma_i$ with the first skew component of $\gamma_{i+1}$ lies in $\U$.
\end{enumerate}
\end{proposition}
\begin{proof}
One direction is trivial.  To prove the other direction, let the left-greedy $\U$-decomposition of $\pi$ be
\[
\pi = \sigma[\alpha_1, \dots, \alpha_m].
\]
Choose the minimum index $i$ such that $|\gamma_i| < |\alpha_i|$.  
By Proposition~\ref{intervals-in-inflation} applied to $\theta[\gamma_1,\dots,\gamma_k]$, the interval $\alpha_i$ of $\pi$ must be of the form
\[
\tau[\gamma_i, \dots, \gamma_j] 
\oplus
\beta_{j+1}
\quad \mbox{or} \quad
\tau[\gamma_i, \dots, \gamma_j] 
\ominus
\beta_{j+1}
\]
for some sum- or skew-prefix $\beta_{j+1}$ of $\gamma_{j+1}$.
If $j > i$ then we already see that the condition (G1) is met.

Otherwise, the values corresponding to $\gamma_i$ (which form an interval) together with the subpermutation of $\beta_{j+1}$ corresponding to the first sum or skew component of $\gamma_{j+1}$ show that either (G2) or (G3) is met.
\end{proof}

We now turn our attention to the substitution closure $\langle \C \rangle$ of a class $\C$. Since the intersection of any family of substitution closed classes is substitution closed, $\langle \C \rangle$ can be defined nonconstructively as the intersection of all substitution closed classes containing $\C$.  For our purposes, the following constructive description is more useful.

\begin{proposition}
\label{prop-subst-completion-const}
The substitution closure of a nonempty class $\C$ is given by
\[
\langle\C\rangle
=
\bigcup_{i=0}^{\infty} \C^{[i]},
\]
where $\C^{[0]}=\{1\}$ and $\C^{[i+1]} = \C[\C^{[i]}]$ for $i\ge0$.
\end{proposition}

\begin{proof}
Since $\langle\C\rangle$ is substitution closed, it contains all $\C^{[i]}$. The other inclusion is proved in a standard way by establishing that
$\bigcup \C^{[i]}$ is substitution closed, and appealing to minimality of $\langle\C\rangle$.
Inflations obey the associative law, i.e. 
$\X[\Y[\Z]]=(\X[\Y])[\Z]$ for any classes $\X,\Y,\Z$.
From this it readily follows that $\C^{[i]}[\C^{[j]}]=\C^{[i+j]}$.
Now consider an inflation $\sigma[\alpha_1,\dots,\alpha_m]$ where $\sigma$ and each $\alpha_i$ are contained in $\bigcup\C^{[i]}$.  Because $\C^{[0]}\subseteq \C^{[1]}\subseteq \C^{[2]}\subseteq\dots$, there is some $k$ such that $\sigma,\alpha_1,\dots,\alpha_m\in\C^{[k]}$.  It then follows that $\sigma[\alpha_1,\dots,\alpha_2]\in\C^{[k]}[\C^{[k]}]=\C^{[2k]}\subseteq\bigcup\C^{[i]}$, as desired.
\end{proof}


Another basic result about substitution closures is the following.

\begin{proposition}
\label{simples-in-substitution-completion}
The substitution closure $\langle \C \rangle$ of a class $\C$ contains exactly the same set of simple permutations as $\C$. 
A permutation class $\D$ is contained in $\langle\C\rangle$ if and only if all simple permutations of $\D$ belong to $\C$.
\end{proposition}

\begin{proof}
The first 
assertion and the forward implication of the second assertion
follow immediately from 
Proposition \ref{prop-subst-completion-const} and
the observation that a simple permutation can never be obtained by a nontrivial inflation.  
For the converse implication of the second assertion
suppose that all simple permutations of a class $\D$ belong to $\C$.
Let $\pi\in\D$ be a nontrivial permutation, and decompose it as $\pi=\sigma[\alpha_1,\dots,\alpha_m]$
where $\sigma$ is simple.
By assumption $\sigma\in\C$, and if we inductively suppose $\alpha_1,\dots,\alpha_m\in\langle\C\rangle$,
we obtain $\pi\in\langle\C\rangle$ by Proposition \ref{prop-subst-completion-const}, as required.
\end{proof}

For the remainder of this section we concern ourselves with the basis of the substitution closure of a class.

\begin{proposition}[Albert and Atkinson~\cite{albert:simple-permutat:}]\label{substitution-completion-basis}
The basis of the substitution closure of a class $\C$ consists of all minimal simple permutations not contained in $\C$.
\end{proposition}

\begin{proof}
Suppose that $\beta \not \in \langle \C \rangle$ is not simple. Then $\beta = \sigma[\alpha_1,\dots,\alpha_m]$ for some simple permutation $\sigma$ and permutations $\alpha_1, \dots, \alpha_m$ all strictly contained in $\beta$. Were $\sigma$ and $\alpha_1, \dots, \alpha_m$ all in $\langle \C \rangle$ we would have $\beta\in\langle\C\rangle$. Hence $\beta$ cannot be a basis element of $\langle \C \rangle$ since all the proper subpermutations of a basis element of a class must lie in the class. Therefore every basis element of $\langle \C \rangle$ is simple, and the proposition follows by the first part of Proposition \ref{simples-in-substitution-completion}. 
\end{proof}

A \emph{parallel alternation} is a permutation whose plot can be divided into two parts, by a single horizontal or vertical line, so that the points on either side of this line are both either increasing or decreasing and for every pair of points from the same part there is a point from the other part which {\it separates\/} them, i.e. lies either horizontally or vertically between them.  It is easy to see that a parallel alternation of length at least four is simple if and only if its length is even and it does not begin with its smallest entry.  Thus there are precisely four simple parallel alternations of each even length at least six, shown in Figure~\ref{fig-par-alts}, and no simple parallel alternations of odd length.

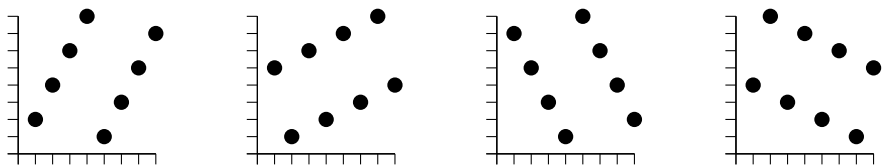
\begin{figure}
\begin{center}
\begin{tabular}{ccccccc}

\psset{xunit=0.009in, yunit=0.009in}
\psset{linewidth=0.005in}
\begin{pspicture}(0,0)(80,80)
\psaxes[dy=10,Dy=1,dx=10,Dx=1,tickstyle=bottom,showorigin=false,labels=none](0,0)(80,80)
\pscircle*(10,20){0.04in}
\pscircle*(20,40){0.04in}
\pscircle*(30,60){0.04in}
\pscircle*(40,80){0.04in}
\pscircle*(50,10){0.04in}
\pscircle*(60,30){0.04in}
\pscircle*(70,50){0.04in}
\pscircle*(80,70){0.04in}
\end{pspicture}

&\rule{0.2in}{0pt}&

\psset{xunit=0.009in, yunit=0.009in}
\psset{linewidth=0.005in}
\begin{pspicture}(0,0)(80,80)
\psaxes[dy=10,Dy=1,dx=10,Dx=1,tickstyle=bottom,showorigin=false,labels=none](0,0)(80,80)
\pscircle*(10,50){0.04in}
\pscircle*(20,10){0.04in}
\pscircle*(30,60){0.04in}
\pscircle*(40,20){0.04in}
\pscircle*(50,70){0.04in}
\pscircle*(60,30){0.04in}
\pscircle*(70,80){0.04in}
\pscircle*(80,40){0.04in}
\end{pspicture}

&\rule{0.2in}{0pt}&

\psset{xunit=0.009in, yunit=0.009in}
\psset{linewidth=0.005in}
\begin{pspicture}(0,0)(80,80)
\psaxes[dy=10,Dy=1,dx=10,Dx=1,tickstyle=bottom,showorigin=false,labels=none](0,0)(80,80)
\pscircle*(10,70){0.04in}
\pscircle*(20,50){0.04in}
\pscircle*(30,30){0.04in}
\pscircle*(40,10){0.04in}
\pscircle*(50,80){0.04in}
\pscircle*(60,60){0.04in}
\pscircle*(70,40){0.04in}
\pscircle*(80,20){0.04in}
\end{pspicture}

&\rule{0.2in}{0pt}&

\psset{xunit=0.009in, yunit=0.009in}
\psset{linewidth=0.005in}
\begin{pspicture}(0,0)(80,80)
\psaxes[dy=10,Dy=1,dx=10,Dx=1,tickstyle=bottom,showorigin=false,labels=none](0,0)(80,80)
\pscircle*(10,40){0.04in}
\pscircle*(20,80){0.04in}
\pscircle*(30,30){0.04in}
\pscircle*(40,70){0.04in}
\pscircle*(50,20){0.04in}
\pscircle*(60,60){0.04in}
\pscircle*(70,10){0.04in}
\pscircle*(80,50){0.04in}
\end{pspicture}

\end{tabular}
\end{center}
\caption{The four orientations of parallel alternations.}
\label{fig-par-alts}
\end{figure}

\begin{theorem}[Schmerl and Trotter~\cite{schmerl:critically-inde:}]\label{thm-schmerl-trotter}
Every simple permutation of length $n\ge 4$ which is not a parallel alternation contains a simple permutation of length $n-1$.  A simple parallel alternation of length $n\geq 4$ contains a simple permutation of length $n-2$.
\end{theorem}

Theorem~\ref{thm-schmerl-trotter} leads rapidly to a sufficient condition for $\langle\C\rangle$ to be finitely based.  Given any permutation class $\C$, we let $\C^{+1}$ denote the class of \emph{one point extensions} of elements of $\C$, i.e., the class of all permutations $\pi$ which contain an entry whose removal yields a permutation in $\C$.

\begin{proposition}\label{prop-vector-simples-basis}
Let $\C$ be a class of permutations.  If $\C^{+1}$ is pwo, then the substitution closure $\langle\C\rangle$ is finitely based.
\end{proposition}

\begin{proof}
The basis elements of $\langle\C\rangle$ are the minimal simple permutations not contained in $\C$ by Proposition~\ref{substitution-completion-basis}.  Clearly there are only finitely many (indeed, at most four) minimal parallel alternations not contained in $\C$.  By Theorem~\ref{thm-schmerl-trotter}, every 
other basis permutation $\beta$ of length $n$ contains a simple permutation $\sigma$ of length $n-1$ which by minimality belongs
to $\C$. Hence $\beta\in\C^{+1}$, and the proposition follows from the assumption that $\C^{+1}$ has no infinite antichains.
\end{proof}

\section{Geometric Grid Classes and Regular Languages}
\label{sec-geomgrid}

We say that a $\zpm$ matrix $M$ of size $t\times u$ is a \emph{partial multiplication matrix} if there exist \emph{column and row signs}
\[
c_1,\ldots,c_t,r_1,\ldots,r_u\in \{1,-1\}
\]
such that every entry $M_{k,\ell}$ is equal to either $c_kr_\ell$ or $0$.  Given a $\zpm$ matrix $M$, we form a new matrix $M^{\times 2}$ by replacing each $0$, $1$, and $-1$ by
\[
\fnmatrix{rr}{0&0\\0&0},
\fnmatrix{rr}{0&1\\1&0}
\mbox{, and }
\fnmatrix{rr}{-1&0\\0&-1},
\]
respectively.  It is easy to see that the standard figure of $M^{\times 2}$ is simply a scaled copy of the standard figure of $M$, and thus $\Geom(M^{\times 2})=\Geom(M)$ for all matrices $M$.  Moreover, the column and row signs $c_k=(-1)^k$, $r_\ell=(-1)^\ell$, show that $M^{\times 2}$ is a partial multiplication matrix, giving the following result.

\begin{proposition}[Albert, Atkinson, Bouvel, Ru\v{s}kuc and Vatter~\cite{albert:geometric-grid-:}]
\label{prop-geom-pmm}
Every geometric grid class is the geometric grid class of a partial multiplication matrix.
\end{proposition}

One useful aspect of geometric grid classes is that they
provide a link between permutations and words.  
Before explaining this connection, we briefly review a few relevant facts about words.  Given a finite alphabet (merely a set of symbols) $\Sigma$, $\Sigma^\ast$ denotes the set of all words (i.e. finite sequences) over $\Sigma$.  The set $\Sigma^\ast$ is partially ordered by the \emph{subword} (or, \emph{subsequence}) order in which $v\le w$ if one can obtain $v$ from $w$ by deleting letters.

Subsets of $\Sigma^\ast$ are called \emph{languages}, and a particular type, \emph{regular languages}, play a central role in our work.  
The empty set, the singleton $\{\emptyword\}$ containing only the empty word, and the singletons $\{a\}$ for each $a\in\Sigma$ are all regular languages; moreover, given two regular languages $K,L\subseteq\Sigma^\ast$, their union $K\cup L$, their concatenation $KL=\{vw\st v\in K\mbox{ and }w\in L\}$, and the star $K^\ast=\{v^{(1)}\cdots v^{(m)}\st v^{(1)},\dots,v^{(m)}\in K\}$ are also regular.  
Every regular language can be obtained by a finite sequence of applications of these rules.
Alternatively, one may define regular languages as those accepted by  finite state automata,
but we will not require this description. 
A language $L$ is \emph{subword closed} if for every $w\in L$ and every subword $v\leq w$ we have $v\in L$.
The \emph{generating function} of the language $L$ is $\sum x^{|w|}$, where the sum is taken over all $w\in L$, and $|w|$ denotes the 
\emph{length} of $w$.  In addition to the above defining properties
of regular languages, we will only require few other basic facts:
\begin{itemize}
\item All finite languages are regular.
\item If $K$ and $L$ are regular languages then so are $K\cap L$ and $K\setminus L$.
\item Every subword closed language is regular.
\item The class of regular languages is closed under homomorphic images and inverse homomorphic images.
\item Every regular language has a rational generating function.
\end{itemize}
For a systematic introduction to regular languages we refer the reader to Hopcroft, Motwani, and Ullman~\cite{hopcroft:introduction-to:a}, or, for a more combinatorial slant, to Flajolet and Sedgewick~\cite[Section I.4 and Appendix A.7]{flajolet:analytic-combin:}.  The regularity of subword closed languages is folkloric, but is specifically proved in Haines~\cite{haines:on-free-monoids:}.

Returning to geometric grid classes, given a partial multiplication matrix $M$ with standard figure $\Lambda$ we define the \emph{cell alphabet} of $M$ as
\[
\Sigma=\{\cell{k}{\ell} \st M_{k,\ell}\neq 0\}.
\]
The permutations in $\Geom(M)$ will be represented, or encoded, by words over $\Sigma$.
Intuitively, the letter $\cell{k}{\ell}$ represents an instruction to place a point in an appropriate position on the line in the $(k,\ell)$ cell of $\Lambda$.  This appropriate position is determined as follows, and the whole process is depicted in Figure~\ref{fig-example-ggc-bij}.

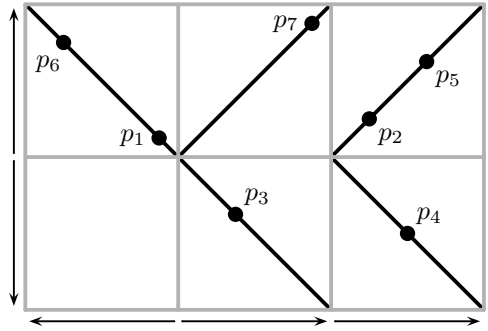
\begin{figure}
\begin{center}
\psset{xunit=0.02in, yunit=0.02in}
\psset{linewidth=0.005in}
\begin{pspicture}(-3,-3)(120,80)
\psline[linecolor=black,linestyle=solid,linewidth=0.02in](0,80)(40,40)
\psline[linecolor=black,linestyle=solid,linewidth=0.02in](40,40)(80,0)
\psline[linecolor=black,linestyle=solid,linewidth=0.02in](40,40)(80,80)
\psline[linecolor=black,linestyle=solid,linewidth=0.02in](80,40)(120,0)
\psline[linecolor=black,linestyle=solid,linewidth=0.02in](80,40)(120,80)
\psline[linecolor=darkgray,linestyle=solid,linewidth=0.02in]{c-c}(0,0)(0,80)
\psline[linecolor=darkgray,linestyle=solid,linewidth=0.02in]{c-c}(40,0)(40,80)
\psline[linecolor=darkgray,linestyle=solid,linewidth=0.02in]{c-c}(80,0)(80,80)
\psline[linecolor=darkgray,linestyle=solid,linewidth=0.02in]{c-c}(120,0)(120,80)
\psline[linecolor=darkgray,linestyle=solid,linewidth=0.02in]{c-c}(0,0)(120,0)
\psline[linecolor=darkgray,linestyle=solid,linewidth=0.02in]{c-c}(0,40)(120,40)
\psline[linecolor=darkgray,linestyle=solid,linewidth=0.02in]{c-c}(0,80)(120,80)
\pscircle*(10,70){0.04in} 
\uput[240](10,70){$p_6$}

\pscircle*(35,45){0.04in} 
\uput[180](35,45){$p_1$}

\pscircle*(55,25){0.04in} 
\uput[45](55,25){$p_3$}

\pscircle*(75,75){0.04in} 
\uput[180](75,75){$p_7$}

\pscircle*(90,50){0.04in} 
\uput[-45](90,50){$p_2$}

\pscircle*(100,20){0.04in} 
\uput[45](100,20){$p_4$}

\pscircle*(105,65){0.04in} 
\uput[-45](105,65){$p_5$}

\psline[linecolor=black,linestyle=solid,linewidth=0.01in,arrowsize=0.05in]{<-c}(-3,1)(-3,39)
\psline[linecolor=black,linestyle=solid,linewidth=0.01in,arrowsize=0.05in]{c->}(-3,41)(-3,79)
\psline[linecolor=black,linestyle=solid,linewidth=0.01in,arrowsize=0.05in]{<-c}(1,-3)(39,-3)
\psline[linecolor=black,linestyle=solid,linewidth=0.01in,arrowsize=0.05in]{c->}(41,-3)(79,-3)
\psline[linecolor=black,linestyle=solid,linewidth=0.01in,arrowsize=0.05in]{c->}(81,-3)(119,-3)
\end{pspicture}
\end{center}
\caption[An example of a geometric grid class.]{In this geometric grid class, with column and row signs as shown, $\bij$ maps the word $\cell{1}{2}\cell{3}{2}\cell{2}{1}\cell{3}{1}\cell{3}{2}\cell{1}{2}\cell{2}{2}$ to $6327415$.}
\label{fig-example-ggc-bij}
\end{figure}

We say that the \emph{base line} of a \emph{column} of $\Lambda$ is the grid line to the left (resp., right) of that column if the corresponding column sign is $1$ (resp., $-1$).  Similarly, the \emph{base line} of a \emph{row} of $\Lambda$ is the grid line below (resp., above) that row if the corresponding row sign is $1$ (resp., $-1$).  We designate the intersection of the two base lines of a cell as its \emph{base point}.  Note that the base point is an endpoint of the line segment of $\Lambda$ lying in this cell.  As this definition indicates, we interpret the column and row signs as specifying the direction in which the columns and rows are `read'.  Owing to this interpretation, we represent the column and row signs in our figures by arrows, as shown in Figure~\ref{fig-example-ggc-bij}.

To every word $w=w_1\cdots w_n\in\Sigma^\ast$ we associate a permutation $\bij(w)$.  First we choose arbitrary distances $0<d_1<\cdots<d_n<1$.  For each $1\le i\le n$, we choose a point $p_i$ corresponding to $w_i$.  Let $w_i=a_{k\ell}$; the point $p_i$ is chosen from the line segment in cell $C_{k,\ell}$, at infinity-norm distance $d_i$ from the base point of this cell.  Finally, $\bij(w)$ denotes the permutation defined by the set $\{p_1,\dots,p_n\}$ of points.

It is a routine exercise to show that $\bij(w)$ does not depend on the particular choice of $d_1,\dots,d_n$, and thus $\bij\st\Sigma^\ast\to \Geom(M)$ is a well-defined mapping.  The basic properties of $\bij$ are described by the following result.

\begin{proposition}[Albert, Atkinson, Bouvel, Ru\v{s}kuc and Vatter~\cite{albert:geometric-grid-:}]
\label{prop-properties-of-bij}
The mapping $\bij$ is length-preserving, finite-to-one, onto, and order-preserving.
\end{proposition}

We then have the following more detailed version of Theorem~\ref{thm-geom-griddable-all-summary}.

\begin{theorem}[Albert, Atkinson, Bouvel, Ru\v{s}kuc and Vatter~\cite{albert:geometric-grid-:}]
\label{thm-geom-griddable-all}
Suppose that $\C\subseteq\Geom(M)$ is a permutation class and $M$ is a partial multiplication matrix with cell alphabet $\Sigma$.  Then the following hold:
\begin{enumerate}
\item[(i)] $\C$ is partially well-ordered.
\item[(ii)] $\C$ is finitely based.
\item[(iii)] There is a regular language $L\subseteq\Sigma^\ast$ such that $\bij$ restricts to a bijection $ L\rightarrow\C$.
\item[(iv)] There is a regular language $L_S$, contained in the regular language from (iii), such that $\bij$ restricts to a bijection between $L_S$ and the simple permutations in $\C$.
\end{enumerate}
\end{theorem}

We also need the following result.

\begin{theorem}[Albert, Atkinson, Bouvel, Ru\v{s}kuc and Vatter~\cite{albert:geometric-grid-:}]
\label{thm-geom-griddable-one-point-extension}
If the class $\C$ is geometrically griddable, then the class $\C^{+1}$ is also geometrically griddable.
\end{theorem}

We end this section with a technical note.  
The mapping $\bij$ `jumbles' the entries, in the sense that the $i$th letter of a word $w\in\Sigma^\ast$
typically does not correspond to the $i$th entry in the permutation $\bij(w)$.
To control for this, we define the \emph{index correspondence} $\psi$ associated to the pair $(\bij,w)$ by letting $\psi(i)$ denote the index of the letter of $w$ which corresponds to the $i$th entry of $\bij(w)$.

\section{Finite Bases and Partial Well-Order}
\label{sec-infinite-finite-bases-pwo}

Two general types of classes will be under investigation in this paper:
\begin{itemize}
\item[(1)]
subclasses of substitution closures of geometric grid classes; and
\item[(2)]
subclasses of inflations of geometric grid classes by strongly rational classes.
\end{itemize}
In this section we establish the pwo property for both these types.
As a consequence we deduce that all classes of type (1) are finitely based.
Note that we cannot hope to have a general finite basis result for type (2), since strongly rational classes need not be
themselves finitely based (see Section \ref{sec-conclusion}).

It follows immediately from Proposition~\ref{prop-vector-simples-basis}, Theorem \ref{thm-geom-griddable-all} (ii) and Theorem~\ref{thm-geom-griddable-one-point-extension} that $\langle\Geom(M)\rangle$ is finitely based.  The basis of $\C\subseteq\langle\Geom(M)\rangle$ therefore consists of an antichain in $\langle\Geom(M)\rangle$ together with, possibly, some of the finitely many basis elements of $\langle\Geom(M)\rangle$ itself.  Therefore we need only prove that $\langle\Geom(M)\rangle$ is pwo.  Morally, owing to the tree-like structure of nested substitutions, this is a consequence of Kruskal's Tree Theorem~\cite{kruskal:well-quasi-orde:}.  However, there are several technical 
issues that would need to be resolved in such an approach,
so we give a proof from first principles.

Given a poset $(P,\le)$, consider the set $P^\ast$ of words with letters from $P$.  
The \emph{generalised subword order} on $P^\ast$ is defined by stipulating that $v=v_1\dots v_k$ is contained in
$w=w_1\dots w_n$  if $w$ has a subsequence $w_{i_1}w_{i_2}\cdots w_{i_k}$ such that $v_j\le w_{i_j}$ for all $j$.
Note that the usual subword ordering on $\Sigma^\ast$ is obtained as a special case where the letters of $\Sigma$ 
are taken to be an antichain.  
We then have the following result from \cite{higman:ordering-by-div:}.

\newtheorem*{higmans-theorem}{\rm\bf Higman's Theorem}
\begin{higmans-theorem}
If $(P,\le)$ is pwo then $P^*$, ordered by the subword order, is also pwo.
\end{higmans-theorem}

We can immediately deduce the pwo property for inflations of geometrically griddable classes.

\begin{proposition}\label{prop-inflate-pwo}
If $\C$ is a geometrically griddable class and $\U$ is a pwo class then the inflation $\C[\U]$ is pwo.
\end{proposition}

\begin{proof}
It suffices to prove that $\Geom(M)[\U]$ is pwo for all partial multiplication matrices $M$.  Suppose that the cell alphabet of $M$ is $\Sigma$ and consider the map
\[
\bij^\U:\left(\Sigma\times\U\right)^\ast\rightarrow\Geom(M)[\U]
\]
which sends $(w_1,\alpha_1)\cdots(w_m,\alpha_m)$ to $\bij(w_1\cdots w_m)[\alpha_{\psi(1)},\dots,\alpha_{\psi(m)}]$ where
$\bij$ is the encoding mapping, and $\psi$ is the index correspondence associated to $(\bij,w)$, both of which have 
been introduced in Section \ref{sec-geomgrid}.
This maps onto $\Geom(M)[\U]$ because $\bij$ maps onto $\Geom(M)$ by Proposition \ref{prop-properties-of-bij}. 
Order  $\left(\Sigma\times\U\right)^\ast$ as follows: $\Sigma\times\U$ is ordered by the direct product ordering, where $\Sigma$ is considered to be an antichain, and then  $\left(\Sigma\times\U\right)^\ast$ is ordered by the generalised subword ordering.
 Using the fact that $\bij$ is order-preserving (Proposition~\ref{prop-properties-of-bij} again), it can be seen that $\bij^\U$ is order-preserving as well.  
Since $\U$ is pwo, $\Sigma\times\U$ is also pwo (as it's simply a union of $|\Sigma|$ copies of $\U$) and thus $\left(\Sigma\times\U\right)^\ast$ is pwo by Higman's Theorem.
It immediately follows that $\Geom(M)[\U]$ is pwo as well.  
\end{proof}

In order to show that substitution closures of geometrically griddable classes are pwo, we borrow a few ideas from the study of posets.  For the purposes of this discussion, we restrict ourselves to posets (such as the poset of all permutations) which are \emph{well-founded}, meaning that they have no infinite strictly decreasing sequences.  Gustedt~\cite{gustedt:finiteness-theo:} defines a partial order on the infinite antichains of a poset, implicit in Nash-Williams~\cite{nash-williams:on-well-quasi-o:}, in which $A\preceq B$ if for every $b\in B$ there exists $a\in A$ such that $a\le b$.  Note that $\preceq$ reverses the set inclusion order: if two infinite antichains satisfy $B\subseteq A$, then $A\preceq B$.

\begin{proposition}[Gustedt~{\cite[Lemma 5]{gustedt:finiteness-theo:}}]\label{prop-minimal-antichain}
For a well-founded poset $P$, the poset of infinite antichains of $P$ under $\preceq$ is also well-founded.
In particular, for every infinite antichain $A\subseteq P$ there is a $\preceq$-minimal infinite antichain $B$
such that $B\preceq A$.
\end{proposition}

\begin{proposition}[Gustedt~{\protect\cite[Theorem 6]{gustedt:finiteness-theo:}}]\label{prop-preceq-min-pcl-pwo}
Suppose that the poset $P$ is well-founded and that the antichain $A$ is $\preceq$-minimal.  Then the \emph{proper closure} of $A$,
\[
A^{\suplessthan}=\{b : b < a\mbox{ for some $a\in A$}\},
\]
is pwo.
\end{proposition}

As an easy consequence we now have:

\begin{theorem}\label{thm-geom-simples-pwo-basis}
If the class $\C$ is geometrically griddable, then every subclass of $\langle\C\rangle$ is finitely based and pwo.
\end{theorem}

\begin{proof}
From our prior discussion, it suffices to prove that $\langle\Geom(M)\rangle$ is pwo for every partial multiplication matrix $M$.  
Suppose $\langle\Geom(M)\rangle$ contains an infinite antichain; then it contains an infinite $\preceq$-minimal antichain $A$ by Proposition~\ref{prop-minimal-antichain}. 
By Proposition~\ref{prop-preceq-min-pcl-pwo} the permutation class $A^{\suplessthan}$ is pwo.
By Proposition~\ref{prop-subst-completion-const} every element $\pi\in A$ can be decomposed as 
$\pi=\sigma[\alpha_1,\dots,\alpha_m]$, where $\sigma\in\Geom(M)$ and $\alpha_1,\dots,\alpha_m$ are properly contained in $\pi$.
In other words, $A\subseteq \Geom(M)[A^{\suplessthan}]$, which cannot happen since $\Geom(M)[A^{\suplessthan}]$
is pwo by Proposition ~\ref{prop-inflate-pwo}.
\end{proof}

In their early investigations of simple permutations, Albert and Atkinson~\cite{albert:simple-permutat:} proved that every permutation class with only finitely many simple permutations is pwo.  Theorem~\ref{thm-geom-simples-pwo-basis} generalises this result, as every finite set of permutations is trivially contained in some geometric grid class.

\section{Properties and Frameworks}\label{sec-frameworks}

In order to establish our enumerative results we adapt ideas introduced by Brignall, Huczynska and Vatter~\cite{brignall:simple-permutat:}.  A {\it property\/} is any set $P$ of permutations, and we say that $\pi$ {\it satisfies\/} $P$ if $\pi\in P$.  Given a family $\P$ of properties and a permutation $\pi$, we write $\P(\pi)$ for the collection of properties in $\P$ satisfied by $\pi$.

In this section we use only two types of properties.  
An \emph{avoidance property} is one of the form $\Av(\beta)$ for some permutation $\beta$.  Note that if $\P$ is a family of avoidance properties and $\sigma\le\pi$, then $\sigma$ must avoid every permutation avoided by $\pi$, so $\P(\sigma)\supseteq\P(\pi)$. 
Additionally we will need the properties $\oplusprop$ and $\ominusprop$, which denote, respectively, the sets of sum decomposable permutations and skew decomposable permutations.

A \emph{$\P$-framework} $\fF$ is a (formal) expression $\sigma[\Q_1,\dots,\Q_m]$ where $\sigma$ is a permutation of length $m$, called the \emph{skeleton} of $\fF$,  and $\Q_i\subseteq\P$ for all $i$. 
We say that $\fF$ \emph{describes} the set of permutations
\[
\{\sigma[\alpha_1,\dots,\alpha_m]\st \mbox{$\P(\alpha_i)=\Q_i$ for all $i$}\}.
\]

Informed by Proposition~\ref{simple-decomp-unique}, we say that a $\P$-framework $\sigma[\Q_1,\dots,\Q_m]$ is \emph{simple} if $\sigma$ is simple and $D_\oplus\notin\Q_1$ (resp., $D_\ominus\notin\Q_1$) if $\sigma=12$ (resp., $\sigma=21$).  We then have the following result.

\begin{proposition}
\label{simple-framework-unique}
If $\P$ is a family of properties containing $\oplusprop$ and $\ominusprop$ then every non-trivial permutation is described by a unique simple $\P$-framework.
\end{proposition}

We say that the $\P$-framework $\sigma[\Q_1,\dots,\Q_m]$ is \emph{nonempty} if it describes at least one permutation; this condition is equivalent to requiring that there be at least one permutation $\alpha_i$ with $\P(\alpha_i)=\Q_i$ for every $i$.

The family $\P$ of properties is \emph{query-complete} if the collection of properties $\P(\sigma[\alpha_1,\dots,\alpha_m])$ is completely determined by $\sigma$ and the collections $\P(\alpha_1)$, $\dots$, $\P(\alpha_m)$.  In other words, $\P$ is query-complete if
\[
\P(\sigma[\alpha_1,\dots,\alpha_m])=\P(\sigma[\alpha_1',\dots,\alpha_m'])
\]
for all permutations
$\sigma$ of length $m$, and all
$m$-tuples $(\alpha_1,\dots,\alpha_m)$ and $(\alpha_1',\dots,\alpha_m')$ which satisfy $\P(\alpha_i)=\P(\alpha_i')$ for all $i$.  When $\P$ is query-complete, we may refer to the \emph{properties} of a nonempty $\P$-framework $\fF$, for which we use the notation $\P(\fF)$, defined as $\P(\pi)$ where $\pi$ is any permutation described by $\fF$.

The situation we are interested in is when $\C\subseteq\langle\Geom(M)\rangle$, i.e., when the simple permutations of $\C$ are contained in a geometric grid class (see Proposition~\ref{simples-in-substitution-completion}). Without loss of generality we will suppose that $M$ is a partial multiplication matrix (Proposition \ref{prop-geom-pmm}).
Let $B$ be the basis of $\C$; recall that $B$ is finite by Theorem \ref{thm-geom-simples-pwo-basis}.
In order to enumerate $\C$, the properties we are interested in are
\[
\P_B=\{\oplusprop,\ominusprop\}\cup\{\Av(\delta)\st \mbox{$\delta\le\beta$ for some $\beta\in B$}\}.
\]
Intuitively, these properties allow us to `monitor', as substitutions are iteratively formed to build $\C\subseteq\langle\Geom(M)\rangle$, `how much' of any basis element from $B$ the resulting permutations contain.  

Let us first verify that $\P_B$ is query-complete; as the union of query-complete sets of properties is again query-complete, we may prove this piece by piece.  First, $\{\oplusprop\}$ is query-complete: $\sigma[\alpha_1,\dots,\alpha_m]\in\oplusprop$ if and only if $\sigma\in\oplusprop$ or $\sigma=1$ and $\alpha_1\in\oplusprop$.  The case of $\{\ominusprop\}$ is similar.  For the rest of $\P_B$, we claim that for every $\beta\in B$, the set $\{\Av(\delta)\st\delta\le\beta\}$ is query-complete.  This is equivalent to stating that knowing the skeleton $\sigma$ and exactly which of the relevant subpermutations of $\beta$ each interval contains allows us to determine whether $\sigma[\alpha_1,\dots,\alpha_m]$ contains a given $\delta\le\beta$; a formal proof is given in Brignall, Huczynska and Vatter~\cite{brignall:simple-permutat:}.

Now let $\P\supseteq\P_B$ be a query-complete set of properties consisting of $\P_B$ together, possibly, with finitely many additional avoidance properties. 
Since $B$ is the relative basis of $\C$ and the properties $\Av(\beta)$ ($\beta\in B$) are in $\P$, it follows that every subset of $\P$ `knows' whether
the permutations it describes belong to $\C$ or not.
More precisely, for a $\P$-framework $\fF$, either every permutation  described by $\fF$ lies in $\C$ or none of them do.
 
The first step of our enumeration of $\C$ is to encode the nonempty, simple $\P$-frameworks which describe permutations in $\C$. 
Let $\Sigma$ be the cell alphabet of $M$, and let $\bij:\Sigma^\ast\rightarrow \Geom(M)$ be the  mapping defined in
Section \ref{sec-geomgrid}.
Since the set $S$ of simple permutations in $\C$ is contained in $\Geom(M)$, Theorem \ref{thm-geom-griddable-all} (iv)
applied to the subclass $\C \cap \Geom(M)$ of $\Geom(M)$ yields a regular language $L_S\subseteq \Sigma^\ast$ such that
$\bij$ induces a bijection between $L_S$ and $S$.
In order to encode $\P$-frameworks, we extend our alphabet to $\Sigma\times 2^{\P}$, that is, ordered pairs whose first component is a letter from $\Sigma$, and whose second component is a subset of $\P$.  
We now define the mapping $\bij^{\P}$ from words in $\left(\Sigma\times2^{\P}\right)^\ast$ to ${\P}$-frameworks with underlying permutations in $\Geom(M)$ by
\[
\bij^{\P}:(w_1,\Q_1)\cdots (w_m,\Q_m)\mapsto \bij(w)[\Q_{\psi(1)},\dots,\Q_{\psi(m)}],
\]
where $w=w_1\cdots w_m$,  and $\psi$ is the index correspondence associated to $(\bij,w)$ defined at the end of Section~\ref{sec-geomgrid}.
Since $\bij$ is onto (Proposition \ref{thm-geom-griddable-all}), so is $\bij^{\P}$.

Given two $\P$-frameworks we write
\[
\tau[\R_1,\dots,R_k]\le\sigma[\Q_1,\dots,Q_m]
\]
if there are indices $1\le i_1<\cdots<i_k\le m$ such that $\tau$ is order isomorphic to $\sigma(i_1)\cdots\sigma(i_k)$ and $\R_j=\Q_{i_j}$ for all $1\le j\le k$.  From Proposition~\ref{prop-properties-of-bij} it follows readily that $\bij^\P$ is order-preserving when considered as a mapping from $\left(\Sigma\times2^{\P}\right)^\ast$ under the subword order to the set of all $\P$-frameworks under the above ordering.

The main result of this section shows that the $\P$-frameworks we are interested in are described by a finite family of
regular languages.

\begin{theorem}
\label{thm-framework-regular}
Let $M$ be a partial multiplication matrix with cell alphabet $\Sigma$, 
let $B$ be any finite set of permutations, 
and let $\P$ be a query-complete set of properties consisting of $\P_B$ together, possibly, with finitely many additional avoidance properties.  For every subset $\Q\subseteq\P$ of properties, there is a regular language $L_\Q\subseteq\left(\Sigma\times2^{\P}\right)^\ast$ such that the mapping $\bij^{\P}$ is a bijection between $L_\Q$ and the nonempty, simple $\P$-frameworks $\fF=\sigma[\Q_1,\dots,\Q_m]$
satisfying $\sigma\in\Geom(M)$ and $\P(\fF)=\Q$.
\end{theorem}

\begin{proof}
Let $L_S\subseteq\Sigma^\ast$ be the regular language such that the mapping $\bij$ is a bijection between $L_S$ and the simple permutations of $\Geom(M)$.  The language
\[
L_S^{\P}
=
\{(w_1,\Q_1)\cdots(w_m,\Q_m)\st w_1\cdots w_m\in L_S\}
\subseteq
\left(\Sigma\times 2^{\P}\right)^\ast
\]
is the inverse image of $L_S$ under the first projection homomorphism,
and is thus regular.

Consider first the case where $\oplusprop\in\Q$.  Thus we must generate all nonempty, simple $\P$-frameworks which describe sum decomposable permutations $\pi$ with $\P(\pi)=\Q$.  
Clearly there are only finitely many such frameworks because they have the form
$12[\Q_1,\Q_2]$.
Choosing a single preimage under $\bij^{\P}$ for each framework yields a finite, and hence regular, set $L_\Q$ with the desired properties.
The case where $\ominusprop\in\Q$ is dual.

Now suppose that $\oplusprop,\ominusprop\notin \Q$.  In this case we must ensure that the $\P$-frameworks we build are neither sum decomposable nor skew decomposable.  This is equivalent to insisting that the skeleton have length at least four.  
Consider the set
$\{\fF \st \Q\subseteq\P(\fF)\}$
of $\P$-frameworks which satisfy at least the properties of $\Q$.
(Note that here we do not require the skeleton be simple -- it can be any element of $\Geom(M)$.)
As all of the properties of $\Q$ are avoidance properties, this set of $\P$-frameworks is closed downward under the $\P$-framework ordering.  
Therefore, because $\bij^\P$ is order-preserving, the set
\[
\{w\in\left(\Sigma\times2^{\P}\right)^\ast \st \Q\subseteq\P(\bij^{\P}(w))\}
\]
is subword-closed, and thus regular.  Dually, the set
\[
\{w\in\left(\Sigma\times2^{\P}\right)^\ast \st \P(\bij^{\P}(w))\subseteq\Q\}
\]
is upward-closed, and thus also regular.  The regular language $L_\Q$ we need to produce is simply the intersection of 
the two sets above (to ensure that $\P(\fF)=\Q$ for every resulting framework $\fF$), 
with the regular language $L_S^{\P}$ (to ensure that the skeleton of $\fF$ is simple),
and the regular language of words of length at least four (to ensure that $\oplusprop,\ominusprop\not\in\P(\fF)$), 
completing the proof.
\end{proof}

\section{Algebraic Generating Functions}
\label{sec-alg-gf}

Our goal now is to utilise Theorem~\ref{thm-framework-regular} ($\P$-frameworks specified by any $\Q\subseteq \P$ are in bijection with a regular language) 
to show that the generating function for a subclass $\C$ of the substitution closure of a geometrically griddable class is algebraic.
It obviously suffices to consider the case where $\C\subseteq\langle\Geom(M)\rangle$.  
Without loss of generality suppose that $M$ is a partial multiplication matrix (Proposition \ref{thm-geom-griddable-all}),
 and denote the corresponding cell alphabet by $\Sigma$.
Let $B$ be the (finite) basis of $\C$.  

For the purposes of this section, it will be helpful to insist that
\[
\P=\P_B\cup\{\Av(21),\Av(12)\},
\]
which is easily seen to be query complete.
With these two extra properties, the family $\Q^\bullet$ consisting of all avoidance properties in $\P$ except $\Av(1)$ satisfies
\[
\P(\pi)=\Q^\bullet
\mbox{ if and only if }
\pi=1.
\]

For every subset $\Q\subseteq \P$ let $f_\Q$ be the generating function for the set
\[
\Delta(\Q)=\{ \pi\in\langle\Geom(M)\rangle \st \P(\pi)=\Q\}
\]
of all permutations in $\langle\Geom(M)\rangle$ described by $\Q$. 
Clearly
\[
\Delta(\Q^\bullet)=\{1\}.
\]
For every other $\Q$ we have
\[
\Delta(\Q)=\bigcup \sigma[\Delta(\Q_1),\dots,\Delta(\Q_m)],
\]
where the (disjoint) union is taken over all simple $\P$-frameworks $\fF=\sigma[\Q_1,\dots,\Q_m]$ with $\sigma\in\Geom(M)$ and $\P(\fF)=\Q$.

This latter set of frameworks is bijectively encoded by the language $L_\Q\subseteq (\Sigma \times 2^\P)^\ast$ via the mapping $\bij^\P$,
as described in Section \ref{sec-frameworks}.
Let $g_\Q$ be the generating function for $L_\Q$ in non-commuting variables representing the letters of our alphabet: 
\begin{equation}
\label{nreq3}
g_\Q=\sum_{w\in L_\Q} w.
\end{equation}
Due to the recursive description of the sets $\Delta(\Q)$ above, 
and the fact that every non-trivial permutation in $\langle\Geom(M)\rangle$ is described by a unique $\P$-framework,
a system of equations for the $f_\Q$ ($\Q\subseteq\P$)
can be obtained by stipulating
\begin{equation}
\label{nreq4}
f_{\Q^\bullet}=x,
\end{equation}
and performing the following substitutions in \eqref{nreq3}:
\begin{equation}
\label{nreq5}
g_\Q\leftarrow f_\Q,\ (u,\R)\leftarrow f_\R\ (\R\subseteq\P).
\end{equation}
The resulting system is finite, although a typical right-hand side of an equation is an infinite series.

On the other hand, the language $L_\Q$ is regular by Theorem~\ref{thm-framework-regular}.
Therefore, as is well known (see Flajolet and Sedgewick~\cite[Proposition I.3]{flajolet:analytic-combin:}),
each $g_\Q$ is the solution of a finite system of linear equations (which almost certainly includes  auxiliary variables).
We then take these systems together and perform substitutions \eqref{nreq5} on them.
The resulting system, together with the equation \eqref{nreq4}, is a finite algebraic system for the $f_\Q$.
We may then perform algebraic elimination (see Flajolet and Sedgewick~\cite[Appendix B.1]{flajolet:analytic-combin:}) to produce a single polynomial equation for each $f_\Q$. 
The generating function $f$ of $\C$ is
$f=\sum f_\Q$,
where the sum is taken over all $\Q$ satisfying $\{ \Av(\beta)\st \beta\in B\}\subseteq \Q\subseteq\P$,
thus proving the following result.

\begin{theorem}
\label{thm-context-free}
Every subclass of the substitution closure of a geometrically griddable class has an algebraic generating function.
\end{theorem}

While we have established Theorem~\ref{thm-context-free} in a purely algebraic manner, it would not be difficult to express our proof in terms of formal languages.
In such an  approach, the above considerations would translate into a proof that the class $\C$
is in bijection with a context-free language, and exhibiting an unambiguous grammar for this language.
Theorem \ref{thm-context-free} would follow from the fact that such languages have algebraic generating functions;
see
Flajolet and Sedgewick~\cite[Proposition I.7]{flajolet:analytic-combin:}.


\section{Inflations by Strongly Rational Classes}

We now consider inflations of the form $\C[\U]$ where $\C$ is geometrically griddable and $\U$ is \emph{strongly rational}, meaning that $\U$ and all its subclasses have rational generating functions.  Recall that $\C[\U]$ is defined as
\[
\C[\U]=\{\sigma[\alpha_1,\dots,\alpha_m]\st\mbox{$\sigma\in\C$ is of length $m$, and $\alpha_1,\dots,\alpha_m\in\U$}\}.
\]
We cannot hope to prove the main result of this section by encoding the permutations of $\C[\U]$ as a regular language, 
simply because we do not know how to encode an arbitrary strongly rational class.  Thus we must consider generating functions for various subsets of $\U$.  The following result is our starting point.

\begin{proposition}[Albert, Atkinson, and Vatter~\cite{albert:subclasses-of-t:}]
\label{prop-strong-rat-indecomps}
If the class $\U$ is strongly rational, then each of the following sets has a rational generating function:
\begin{itemize}
\item the sum indecomposable permutations in $\U$;
\item the sum decomposable permutations in $\U$;
\item the skew indecomposable permutations in $\U$;
\item the skew decomposable permutations in $\U$; and
\item the permutations in $\U$ which are both sum and skew indecomposable.
\end{itemize}
\end{proposition}

As in Section~\ref{sec-frameworks}, given a finite set $B$ of permutations, we define the family of properties $\P_B$ by
\[
\P_B=\{\oplusprop,\ominusprop\}\cup\{\Av(\delta)\st \mbox{$\delta\le\beta$ for some $\beta\in B$}\}.
\]

\begin{proposition}
\label{prop-strong-rat-properties}
Let $\U$ be a strongly rational permutation class, and let $B$ be a finite set of permutations.  For every subset $\Q\subseteq\P_B$ of properties, the generating function for the permutations in $\U$ satisfying $\P_B(\pi)=\Q$ is rational.
\end{proposition}

\begin{proof}
Let $g_\Q$ denote the generating function for the permutations we want to count, i.e., the permutations in $\U$ which satisfy precisely the properties $\Q$.  Further, given a set $\R\subseteq\P_B$ of properties, let $f_\R$ denote the generating function for the permutations in $\U$ which satisfy at least the properties of $\R$, but possibly more.  
Because $\P_B$ consists of the properties of being sum- and skew decomposable, together with a collection of avoidance properties,
each $f_\R$ corresponds to one of the bullet points in Proposition~\ref{prop-strong-rat-indecomps} for a subclass of $\U$.
Specifically, 
letting $B'=\{ \delta\st \Av(\delta)\in\R\}$ and $\V=\U\cap \Av(B')$, we have:
\begin{itemize}
\item
if $\oplusprop,\ominusprop\not\in\R$ then $f_\R$ is the generating function for the class $\V$;
\item
if $\oplusprop\in\R$ and $\ominusprop\not\in\R$ then $f_\R$ is the generating function for the sum decomposable permutations in $\V$;
\item
if $\oplusprop\not\in\R$ and $\ominusprop\in\R$ then $f_\R$ is the generating function for the skew decomposable permutations in $\V$;
\item
if $\oplusprop,\ominusprop\in\R$ then $f_\R=0$.
\end{itemize}
In any case, $f_\R$ is rational.
To complete the proof, we need only note that
\[
g_\Q
=
\sum_{\R\st\Q\subseteq\R\subseteq\P_B}
(-1)^{|\R\setminus\Q|} f_\R
\]
by inclusion-exclusion.
\end{proof}

Our argument that inflations of geometrically griddable classes by strongly rational classes are strongly rational (Theorem~\ref{thm-geom-inflate-enum}) is fairly technical,
but the underlying idea is quite simple:
Given such a class $\D\subseteq \C[\U]$, where $\C$ is geometrically griddable and $\U$ is strongly rational,
we find a suitable set of properties $\P$ so that we can encode all the requisite $\P$-frameworks by a regular language $L_\D$.
Then we use a variant of Proposition \ref{prop-strong-rat-properties}
to show that the generating functions for permutations in $\U$ described by arbitrary $\Q\subseteq\P$ are rational.
Finally, we substitute these rational generating functions into the rational generating function for the language $L_\D$,
yielding a rational generating function for $\D$.

There are two major obstacles to this programme.
The first is that for obvious reasons we need our set of properties to discriminate between inflations $\sigma[\alpha_1,\dots,\alpha_m]$
that belong to $\D$ and those that don't, and also, for technical reasons which will become apparent shortly,
between the inflations that belong to $\U$ and those that don't.
But we cannot assume that $\U$ (and hence $\D$) are finitely based and then use the basis permutations to construct $\P$.
Fortunately, the pwo property comes to the rescue (Proposition \ref{prop-inflate-pwo}), and we can use \emph{relative bases} instead.
Specifically, let $B_\D$ (respectively $B_\U$) be the set of all basis elements of $\D$ (respectively, $\U$) that belong to $\C[\U]$;
note that $B_\D$ and $B_\U$ are both finite because $\C[\U]$ is pwo, and that
\begin{eqnarray*}
\D&=&\C[\U]\cap \Av(B_\D),\\
\U&=&\C[\U]\cap \Av(B_\U).
\end{eqnarray*}
We let $B=B_\D\cup B_\U$, and construct the set of properties $\P_B$ as in Section \ref{sec-frameworks}.

The second obstacle is that
as it stands, the family of properties $\P_B$ is still not sufficiently discriminating.
Indeed, Proposition~\ref{prop-left-greedy-U-inflate} demonstrates that a single $\P_B$-framework may well describe both left-greedy and non-left-greedy $\U$-inflations.  
Consider, for example, a $\P_B$-framework of the form $12[\Q_1,\Q_2]$.  Some $\U$-inflations described by this $\P_B$-framework will be left-greedy (if the first sum component of the second interval cannot `slide' to the first interval), while the others will not be.  To address this issue, we say that the \emph{first component} of the permutation $\pi$ is the first sum component of $\pi$ if $\pi$ is sum decomposable, the first skew component of $\pi$ if $\pi$ is skew decomposable, and $\pi$ itself otherwise (if $\pi$ is neither sum nor skew decomposable).  Observe that this notion is well-defined because no permutation is both sum and skew decomposable.  We can now introduce the \emph{first component avoidance properties}:
\[
\Av^\initcomp(\delta)
=
\{\pi \st \mbox{the first component of $\pi$ avoids $\delta$}\}.
\]
We need the full range of these properties,
\[
\P_B^\initcomp
=
\{\Av^\initcomp(\delta) \st \mbox{$\delta\le\beta$ for some $\beta\in B$}\}.
\]

The enlarged family of properties $\tilde{\P}_B=\P_B\cup\P_B^\initcomp$ is query-complete. 
Indeed, it suffices to show that the properties $\P_B^\initcomp$ of $\sigma[\Q_1,\dots,\Q_m]$ are completely determined by $\sigma$ and the sets $\Q_i\cap\P_B$ of properties.  If $\sigma=\tau\oplus\xi$ for a sum indecomposable $\tau$ and nonempty $\xi$, we see that
\[
\Av^\initcomp(\delta)\in\P_B^\initcomp(\sigma[\Q_1,\dots,\Q_m])
\mbox{ if and only if }
\Av(\delta)\in\P_B(\tau[\Q_1,\dots,\Q_{|\tau|}]),
\]
which can be determined from $\tau$ and $\Q_1,\dots,\Q_{|\tau|}$ because $\P_B$ is query-complete.  
The analogous assertion holds when $\sigma=\tau\ominus\xi$ for a skew indecomposable $\tau$ and nonempty $\xi$.  
If $\sigma$ is neither sum nor skew decomposable, then the criterion is even simpler:
\[
\Av^\initcomp(\delta)\in\P_B^\initcomp(\sigma[\Q_1,\dots,\Q_m])
\mbox{ if and only if }
\Av(\delta)\in\P_B(\sigma[\Q_1,\dots,\Q_m]),
\]
which again can be determined from $\sigma$ and $\Q_1,\dots,\Q_m$ because $\P_B$ is query-complete.

We begin by reiterating that, because $B$ contains the relative bases for $\D$ and $\U$ in $\C[\U]$,
the $\tilde{\P}_B$ frameworks respect the boundary between each of these two classes and its complement in $\C[\U]$.

\begin{proposition}
\label{nrlemma1}
Given a $\tilde{\P}_B$-framework $\sigma[\Q_1,\dots,\Q_m]$ with $\sigma\in\C$, either all $\U$-inflations described by it lie in $\D$ (respectively $\U$) or none do.
\end{proposition}

Next we show that $\tilde{\P}_B$-frameworks can also distinguish between left-greedy and non-left-greedy $\U$-inflations.

\begin{proposition}
\label{lem-firstcomp-capture}
Given a $\tilde{\P}_B$-framework $\sigma[\Q_1,\dots,\Q_m]$ with $\sigma\in\C$, either all $\U$-inflations described by it are left-greedy or none are.
\end{proposition}

\begin{proof}
We need to show that either every $\U$-inflation described by $\sigma[\Q_1,\dots,\Q_m]$ satisfies one of the three conditions (G1)--(G3) of Proposition~\ref{prop-left-greedy-U-inflate} or that none do.  Suppose that some $\U$-inflation, say $\pi=\sigma[\alpha_1,\dots,\alpha_m]$, described by $\sigma[\Q_1,\dots,\Q_m]$ satisfies at least one of these conditions.  
If this condition is (G1), the assertion follows from Proposition~\ref{nrlemma1}. Now suppose that (G1) is not satisfied, and that (G2) is.
Thus $\sigma$ contains an increasing run $\sigma(i+1)=\sigma(i)+1$ and the sum of $\alpha_i$ and the first sum component of $\alpha_{i+1}$ lies in $\U$. 
Furthermore, since (G1) does not hold, this first sum component is not the entire $\alpha_{i+1}$.
Translating to our properties, this will happen if and only if $\oplusprop\in \Q_{i+1}$, $\ominusprop\not\in\Q_{i+1}$, and
for all $\delta_1,\delta_2$ such that $\Av(\delta_1)\not\in\Q_i\cap \P_B$ and $\Av^\initcomp(\delta_2)\not\in\Q_{i+1}\cap\P_B^\initcomp$
we have $\delta_1\oplus\delta_2\not\in B_\U$.
Therefore (G2) will hold for $\pi$ if and only if it holds for all $\U$-inflations described by $\sigma[\Q_1,\dots,\Q_m]$. A similar argument applies in the case that (G3) is satisfied but (G1) is not, completing the proof.
\end{proof}

Proposition~\ref{lem-firstcomp-capture} allows us to call a non-empty $\tilde{\P}_B$-framework \emph{left-greedy} if every $\U$-inflation it describes is left-greedy.  The price we pay for this additional discriminating power of $\tilde{\P}_B$ is that we must strengthen Proposition~\ref{prop-strong-rat-properties} to include first component properties.

\begin{proposition}
\label{prop-strong-rat-properties-firstcomp}
Let $\U$ be a strongly rational permutation class and $B$ be a finite set of permutations.  For every subset $\Q\subseteq\tilde{\P}_B$ of properties, the generating function for the set of permutations $\pi\in\U$ satisfying $\tilde{\P}_B(\pi)=\Q$ is rational.
\end{proposition}

\begin{proof}
For any set $\S\subseteq\P_B$ of properties, let $g_\S$ denote the generating function for permutations in $\U$ satisfying $\P_B(\pi)=\S$; all such generating functions are rational by Proposition~\ref{prop-strong-rat-properties}.  Further define
\[
\R=\{\Av(\delta)\st \Av^\initcomp(\delta)\in\Q\},
\]
and let $f$ denote the generating function of permutations in $\U$ satisfying $\tilde{\P}_B(\pi)=\Q$.  There are three cases to consider.

First suppose that $\oplusprop,\ominusprop\notin\Q$.  Thus if $\pi$ satisfies $\tilde{\P}_B(\pi)=\Q$ then the first component of $\pi$ is $\pi$ itself, so $\pi\in\Av^\initcomp(\delta)$ if and only if $\pi\in\Av(\delta)$.  Therefore $f=0$ unless $\R=\Q\cap\P_B$, in which case $f=g_{\Q\cap\P_B}$, which is rational.

Now suppose that $\oplusprop\in\Q$, so we aim to count sum decomposable permutations.  If $\ominusprop\in\Q$ then $f=0$, so we may assume that $\ominusprop\notin\Q$.  We now see that $f$ counts permutations of the form $\sigma\oplus\tau$ where $\P_B(\sigma)$ is equal to $\R$ or to $\R'=\R\cup\{\ominusprop\}$, and $\P_B(\sigma\oplus\tau)=\Q\cap\P_B=\T$.  Thus we have that
\[
f
=
\sum_{\S\st\P_B(12[\R,\S])=\T} g_{\R}g_{\S}
+
\sum_{\S\st\P_B(12[\R',\S])=\T} g_{\R'}g_{\S},
\]
which is also rational.  The case where $\ominusprop\in\Q$ is analogous, completing the proof.
\end{proof}

A crucial step in our argument is a regular encoding of left-greedy $\tilde{\P}_B$-frameworks.  While we know (in principle) how to recognise a single left-greedy framework, this does not help to encode them all simultaneously.  
To address this issue, we adapt the marking technique from Albert, Atkinson, Bouvel, Ru\v{s}kuc, and Vatter~\cite{albert:geometric-grid-:}.  A \emph{marked permutation} is a permutation in which the entries are allowed to be marked, which we designate with an overline.  
The intention of a marking is to highlight some special characteristic of the marked entries.

A \emph{marked $\P$-framework} is then a $\P$-framework $\sigma[\Q_1,\dots,\Q_m]$ in which the skeleton $\sigma$ is marked.  
The mapping $\bij^{\P}$
defined in Section \ref{sec-frameworks}
can be extended in a natural manner to a mapping $\overline{\bij}^{\P}$ from
$\left(\left(\Sigma\times 2^{\P}\right)\cup \left(\overline{\Sigma}\times 2^{\P}\right)\right)^\ast$
to the set of marked $\P$-frameworks $\sigma[\Q_1,\dots,\Q_m]$ with $\sigma\in\C$;
here $\overline{\Sigma}=\{\overline{a}\st a\in\Sigma\}$ is the \emph{marked cell alphabet}, and $\overline{\bij}^{\P}$ maps marked letters to marked entries.

We can extend the order on $\P$-frameworks defined in Section~\ref{sec-frameworks} to this context as follows: for two marked frameworks we write
\[
\tau[\R_1,\dots,\R_k]\le\sigma[\Q_1,\dots,\Q_m]
\]
if there are indices $1\le i_1<\cdots<i_k\le m$ such that
\begin{itemize}
\item $\sigma(i_1),\sigma(i_2),\dots,\sigma(i_k)$ is order isomorphic to $\tau$ (as ordinary, unmarked, permutations); and
\item for all $1\le j\le k$, $\sigma(i_j)$ is marked if and only if $\tau(j)$ is marked; and
\item for all $1\le j\le k$, $\R_j=\Q_{i_j}$.
\end{itemize}
With this order, it follows from Proposition~\ref{prop-properties-of-bij} that the mapping $\overline{\bij}^{\P}$ is order-preserving.

\begin{theorem}
\label{thm-geom-inflate-enum}
The class $\C[\U]$ is strongly rational for all geometrically griddable classes $\C$ and strongly rational classes $\U$.
\end{theorem}
\begin{proof}
We retain the set-up introduced so far, so $\C\subseteq\Geom(M)$ for a partial multiplication matrix $M$ with cell alphabet $\Sigma$, $\D$ is a subclass of $\C[\U]$, and $B=B_\D\cup B_\U$ where $B_\D$ and $B_\U$ denote, respectively, the relative bases of $\D$ and $\U$ in $\C[\U]$.

We now seek to 
encode the nonempty, left-greedy $\tilde{\P}_B$-frameworks by means of a regular language
and the mapping $\bij^{\tilde{\P}_B}$.
To do this we mark an interval of $\sigma$ which might satisfy one of the conditions (G1)--(G3) of Proposition~\ref{prop-left-greedy-U-inflate}.  To this end, we say that a marking of a $\tilde{\P}_B$-framework $\sigma[\Q_1,\dots,\Q_m]$ is \emph{threatening} if the marked entries constitute a (possibly trivial) interval of $\sigma$ given by the indices $\interval{i}{j}$ which is order isomorphic to $\tau$, and either
\begin{itemize}
\item the permutations described by $\tau[\Q_i,\Q_{i+1},\dots,\Q_j]$ lie in $\U$ (corresponding to (G1)); or
\item $|\tau|=2$ and $\tau[\Q_i,\Q_{i+1}]$ is not a left-greedy $\tilde{\P}_B$-framework (corresponding to (G2) and (G3)).
\end{itemize}
Note that every marked $\tilde{\P}_B$-framework with zero, one, or all marked letters is threatening, and thus every $\tilde{\P}_B$-framework has several threatening markings.  However, if a $\tilde{\P}_B$-framework has a threatening marking with two or more but not all marked letters, then that $\tilde{\P}_B$-framework is not left-greedy.  Therefore, the left-greedy $\tilde{\P}_B$-frameworks are precisely the $\tilde{\P}_B$-frameworks which do not have such markings, and our goal is to identify them.

Importantly, given a threateningly marked $\tilde{\P}_B$-framework $\sigma[\Q_1,\dots,\Q_m]$ which describes permutations from $\D$
(recall Proposition~\ref{nrlemma1}), if $\tau[\R_1,\dots,\R_k]\le\sigma[\Q_1,\dots,\Q_m]$ in the order defined above, then $\tau[\R_1,\dots,\R_k]$ 
is also threateningly marked (and also describes permutations in $\D$).  
Therefore the set of all threateningly marked $\tilde{\P}_B$-frameworks is a downset.
Since $\overline{\bij}^{\tilde{\P}_B}$ is order-preserving, the pre-image of this downset
\[
\overline{J}=
\left\{
w\in\left(\left(\Sigma\times 2^{\P}\right)\cup \left(\overline{\Sigma}\times 2^{\P}\right)\right)^\ast
\st
\overline{\bij}^{\tilde{\P}_B}(w)
\mbox{ is threateningly marked}
\right\}
\]
is subword-closed, and thus regular.

Now let
\[
\Gamma
\st
\left(\left(\Sigma\times 2^{\tilde{\P}_B}\right)\cup \left(\overline{\Sigma}\times 2^{\tilde{\P}_B}\right)\right)^\ast
\rightarrow
\left(\Sigma\times 2^{\tilde{\P}_B}\right)^\ast
\]
denote the homomorphism which removes markings.  Because every $\tilde{\P}_B$-framework has a threatening marking, $\Gamma(\overline{J})$ encodes all $\tilde{\P}_B$-frameworks.  We want to remove from $\Gamma(\overline{J})$ the set of non-left-greedy $\tilde{\P}_B$-frameworks.  These non-left-greedy frameworks are precisely the frameworks which have a marked encoding in $\overline{J}\cap\overline{K}$ where
\[
\overline{K}
=
\left\{
\begin{array}{l}
\mbox{words in $\left(\left(\Sigma\times 2^{\tilde{\P}_B}\right)\cup \left(\overline{\Sigma}\times 2^{\tilde{\P}_B}\right)\right)^\ast$ with at}\\
\mbox{least two marked and one unmarked letters}
\end{array}\right\}.
\]
The language $\overline{K}$ is clearly regular.  Furthermore,  $\Gamma(\overline{J})$ and $\Gamma(\overline{J}\cap\overline{K})$ are both regular
as homomorphic images of regular languages.  Therefore the language 
$L_\D=\Gamma(\overline{J})\setminus\Gamma(\overline{J}\cap\overline{K})$ is regular, and it
encodes nonempty, left-greedy $\tilde{\P}_B$-frameworks describing permutations from $\D$.

Recall that every permutation $\pi\in\D$ has a unique left-greedy $\U$-decomposition $\pi=\sigma[\alpha_1,\dots,\alpha_m]$
with $\sigma\in\C$ and $\alpha_1,\dots,\alpha_m\in\U$.
Furthermore, recall that every $\alpha\in\U$ is described by a unique $\tilde{\P}_B$ framework.
Therefore, a generating function for the class $\D$ is obtained by taking 
the generating function $g=\sum_{w\in L_\D} w$ in non-commuting variables representing the letters of our alphabet, 
and substituting for each letter $(u,\Q)$ the generating function $f_\Q$
for the set of all permutations in $\U$ described by $\Q$.
The function $g$ is rational because $L_\D$ is a regular language,
and the functions $f_\Q$ are rational by Proposition~\ref{prop-strong-rat-properties-firstcomp}.
It follows that $\D$ itself has a rational generating function, and the theorem is proved.
\end{proof}

\section{Small Permutation Classes}
\label{sec-spc-rational}

With Theorem~\ref{thm-geom-inflate-enum}, we have all the enumerative machinery we need to prove that small permutation classes are strongly rational, but we must spend a bit of time beforehand aligning the results of Vatter~\cite{vatter:small-permutati:} with those of this paper.     

One of the biggest differences between the two approaches is that the grid classes we have discussed so far are much 
more constrained than the generalised grid classes of Vatter~\cite{vatter:small-permutati:}.  Suppose that $\M$ is a $t\times u$ matrix of permutation classes (we use calligraphic font for matrices containing permutation classes).  An {\it $\M$-gridding\/} of the permutation $\pi$ of length $n$ in this context is a pair of sequences $1=c_1\le\cdots\le c_{t+1}=n+1$ (the column divisions) and $1=r_1\le\cdots\le r_{u+1}=n+1$ (the row divisions) such that for all $1\le k\le t$ and $1\le\ell\le u$, the entries of $\pi$ from indices $c_k$ up to but not including $c_{k+1}$, which have values from $r_{\ell}$ up to but not including $r_{\ell+1}$ are either empty or order isomorphic to an element of $\M_{k,\ell}$.  The {\it grid class of $\M$\/}, written $\Grid(\M)$, consists of all permutations which possess an $\M$-gridding.  Furthermore, we say that the permutation class $\C$ is {\it $\D$-griddable\/} if $\C\subseteq\Grid(\M)$ for some (finite) matrix $\M$ whose entries are all equal to $\D$.

Between these generalised grid classes and the geometric grid classes we have been considering lie the \emph{monotone grid classes}, of the form $\Grid(\M)$ for a matrix $\M$ whose entries are restricted to $\emptyset$, $\Av(21)$, and $\Av(12)$.  When considering monotone grid classes we abbreviate these three classes to $0$, $1$, and $-1$ (respectively).  The class $\C$ is \emph{monotone griddable} if $\C$ lies in $\Grid(M)$ for some $\zpm$ matrix $M$.

To explain the relationship between monotone and geometric grid classes we need to introduce a graph.  The \emph{row-column graph} of a $t\times u$ matrix $M$ is the bipartite graph on the vertices $x_1$, $\dots$, $x_t$, $y_1$, $\dots$, $y_u$ where $x_k\sim y_\ell$ if and only if $M_{k,\ell}\neq 0$.  It can then be shown (see Albert, Atkinson, Bouvel, Ru\v{s}kuc, and Vatter~{\cite[Theorem 3.2]{albert:geometric-grid-:}}) that $\Geom(M)=\Grid(M)$ if and only if the row-column graph of $M$ is a forest.

In order to use grid classes to describe small permutation classes one needs the following generalisation of a result of Huczynska and Vatter~\cite{huczynska:grid-classes-an:}.

\begin{theorem}[Vatter~{\cite[Theorem 3.1]{vatter:small-permutati:}}]
\label{gridding-characterization}
A permutation class is $\D$-griddable if and only if it does not contain arbitrarily long sums or skew sums of basis elements of $\D$.
\end{theorem}

We now introduce a specific class.  The \emph{increasing oscillating sequence} is the infinite sequence defined by
\[
4,1,6,3,8,5,\dots,2k+2,2k-1,\dots
\]
(which contains every positive integer except $2$).  An {\it increasing oscillation\/} is any sum indecomposable permutation that is contained in the increasing oscillating sequence (this term dates back to at least Pratt~\cite{pratt:computing-permu:}).  A \emph{decreasing oscillation} is the reverse of an increasing oscillation, and collectively these permutations are called \emph{oscillations}.


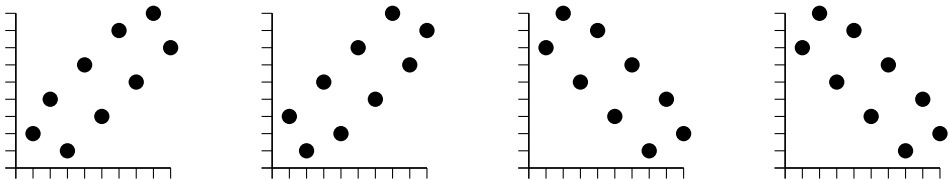
\begin{figure}
\begin{center}
\begin{tabular}{ccccccc}

\psset{xunit=0.009in, yunit=0.009in}
\psset{linewidth=0.005in}
\begin{pspicture}(0,0)(90,90)
\psaxes[dy=10,Dy=1,dx=10,Dx=1,tickstyle=bottom,showorigin=false,labels=none](0,0)(90,90)
\pscircle*(10,20){0.04in}
\pscircle*(20,40){0.04in}
\pscircle*(30,10){0.04in}
\pscircle*(40,60){0.04in}
\pscircle*(50,30){0.04in}
\pscircle*(60,80){0.04in}
\pscircle*(70,50){0.04in}
\pscircle*(80,90){0.04in}
\pscircle*(90,70){0.04in}
\end{pspicture}

&\rule{0.2in}{0pt}&

\psset{xunit=0.009in, yunit=0.009in}
\psset{linewidth=0.005in}
\begin{pspicture}(0,0)(90,90)
\psaxes[dy=10,Dy=1,dx=10,Dx=1,tickstyle=bottom,showorigin=false,labels=none](0,0)(90,90)
\pscircle*(10,30){0.04in}
\pscircle*(20,10){0.04in}
\pscircle*(30,50){0.04in}
\pscircle*(40,20){0.04in}
\pscircle*(50,70){0.04in}
\pscircle*(60,40){0.04in}
\pscircle*(70,90){0.04in}
\pscircle*(80,60){0.04in}
\pscircle*(90,80){0.04in}
\end{pspicture}

&\rule{0.2in}{0pt}&

\psset{xunit=0.009in, yunit=0.009in}
\psset{linewidth=0.005in}
\begin{pspicture}(0,0)(90,90)
\psaxes[dy=10,Dy=1,dx=10,Dx=1,tickstyle=bottom,showorigin=false,labels=none](0,0)(90,90)
\pscircle*(10,70){0.04in}
\pscircle*(20,90){0.04in}
\pscircle*(30,50){0.04in}
\pscircle*(40,80){0.04in}
\pscircle*(50,30){0.04in}
\pscircle*(60,60){0.04in}
\pscircle*(70,10){0.04in}
\pscircle*(80,40){0.04in}
\pscircle*(90,20){0.04in}
\end{pspicture}

&\rule{0.2in}{0pt}&

\psset{xunit=0.009in, yunit=0.009in}
\psset{linewidth=0.005in}
\begin{pspicture}(0,0)(90,90)
\psaxes[dy=10,Dy=1,dx=10,Dx=1,tickstyle=bottom,showorigin=false,labels=none](0,0)(90,90)
\pscircle*(10,70){0.04in}
\pscircle*(20,90){0.04in}
\pscircle*(30,50){0.04in}
\pscircle*(40,80){0.04in}
\pscircle*(50,30){0.04in}
\pscircle*(60,60){0.04in}
\pscircle*(70,10){0.04in}
\pscircle*(80,40){0.04in}
\pscircle*(90,20){0.04in}
\end{pspicture}

\end{tabular}
\end{center}
\caption{The four oscillations of length $9$.}\label{fig-osc}
\end{figure}

We let $\O$ denote the \emph{downward closure} of the set of (increasing and decreasing) oscillations; in other words, $\O$ consists of 
all oscillations and their subpermutations.  Further let $\O_k$ denote the downward closure of the set of oscillations of length at most $k$;
this is a finite class.  Using Theorem~\ref{gridding-characterization} and Schmerl and Trotter's Theorem~\ref{thm-schmerl-trotter}, Vatter~\cite{vatter:small-permutati:} showed via a computational argument that every small permutation class is $\langle\O\rangle$-griddable.

In fact, a much stronger result holds.  It can be shown that the growth rate of the downward closure of the set of increasing oscillations is precisely equal to $\kappa$.  Therefore, if $\C$ contains all increasing oscillations, it is not small.  Moreover, the increasing oscillations `almost' form a chain, and so if a class does not contain one increasing oscillation, there is a bound on the length of the increasing oscillations it can contain.  By symmetry, every small permutation class also has a bound on the length of decreasing oscillations it can contain, and thus every small permutation class is actually $\langle\O_k\rangle$-griddable for some integer $k$.

Because classes containing permutations with complicated substitution decompositions can be shown to have growth rates greater than $\kappa$ (via another computational argument), we can say more about the griddability of small permutation classes.  First, define the class $\tilde{\O}_k$ by
\[
\tilde{\O}_k=\O_k\cup\Av(21)\cup\Av(12),
\]
i.e., the oscillations of length at most $k$ together with the monotone permutations of all lengths.  Via a minor translation in notation, 
and recalling the $\C^{[d]}$ construction from Proposition \ref{prop-subst-completion-const}, we quote the following result.

\begin{theorem}[Vatter~{\cite[Theorems 4.3 and 4.4]{vatter:small-permutati:}}]
\label{thm-subst-gridding-main}
Every small permutation class is $\tilde{\O}_k^{[d]}$-griddable for some choice of integers $k$ and $d$.
\end{theorem}

The restriction to $\tilde{\O}_k^{[d]}$-griddings is important for two reasons.  The first is that these classes are strongly rational.  
Indeed, by iterating Theorem~\ref{thm-geom-inflate-enum}, we obtain the following.

\begin{corollary}
\label{cor-bounded-depth-strong-rat}
If the class $\C$ is geometrically griddable, then the class $\C^{[d]}$ is strongly rational for every $d$.
\end{corollary}

Clearly $\tilde{\O}_k$, which contains only finitely many nonmonotone permutations, is  geometrically griddable.  Therefore Corollary~\ref{cor-bounded-depth-strong-rat} implies that $\tilde{\O}_k^{[d]}$ is strongly rational for all $d$ and $k$.

The second benefit of the restriction to $\tilde{\O}_k^{[d]}$-griddings is that, because $\tilde{\O}_k^{[d]}$ contains neither long simple permutations nor complicated substitution decompositions, a quite technical argument allows us to `slice' the $\tilde{\O}_k^{[d]}$-griddings of small permutation classes, as formalised below.

\begin{theorem}[Vatter~{\cite[Theorem 5.4]{vatter:small-permutati:}}]
\label{small-classes-gridding}
Every small permutation class is $M$-griddable for a matrix $M$ in which:
\begin{enumerate}
\item[(S1)] every entry is $\tilde{\O}_k^{[d]}$, $\Av(21)$, $\Av(12)$, or the empty set;
\item[(S2)] every entry equal to $\tilde{\O}_k^{[d]}$ is the unique nonempty entry in its row and column; and
\item[(S3)] if two nonempty entries share a row or a column with each other (in which case they both must be monotone by (S2)), then neither shares a row or column with any other nonempty entry.
\end{enumerate}
\end{theorem}

Condition (S2) shows that every small permutation class is $M$-griddable for a matrix $M$ in which every pair of `interacting' cells is monotone.  Therefore, if $M$ contains any nonmonotone cells, we may as well view them as inflations of a singleton cell, or indeed, of any type of monotone cell at all.  We can express this consequence of Theorem~\ref{small-classes-gridding} in the language of monotone grid classes by saying that every small permutation class is contained in $\Grid(M)[\tilde{\O}_k^{[d]}]$ for some $\zpm$ matrix $M$ and integers $k$ and $d$.  Furthermore, condition (S3) implies that this matrix $M$ can be taken so that its row-column graph is a forest, so $\Grid(M)=\Geom(M)$.  Therefore we see that, for every small permutation class $\C$, there is a $\zpm$ matrix $M$ and integers $k$ and $d$ such that
\[
\C\subseteq\Geom(M)[\tilde{\O}_k^{[d]}].
\]
From here, we need only apply Theorem~\ref{thm-geom-inflate-enum} to establish the desired result.

\begin{theorem}
\label{thm-small-rational}
All small permutation classes have rational generating functions.
\end{theorem}

\section{Conclusion}
\label{sec-conclusion}

We have extended the substitution decomposition to handle enumeration far beyond the initial investigations of Albert and Atkinson~\cite{albert:simple-permutat:}, to the point where these techniques apply to all permutation classes of growth rate less than $\kappa\approx 2.20557$.  Still, it is worth reflecting on how difficult the enumeration of permutation classes remains.  Over fifteen years ago Noonan and Zeilberger~\cite{noonan:the-enumeration:} suggested that every finitely based permutation class has a holonomic generating function.  Roughly ten years after that, Zeilberger (see Elder and Vatter~\cite{elder:problems-and-co:}) conjectured precisely the opposite, in fact specifying a potential counterexample: he speculated that $\Av(1324)$ might not have a holonomic generating function.

Perhaps even if the generating functions of permutation classes are not well behaved, their growth rates might be.  Balogh, Bollob\'as, and Morris~\cite{balogh:hereditary-prop:ordgraphs} were overly optimistic in this direction: they made a conjecture whose truth would have implied that all growth rates of permutation classes are algebraic numbers, which was disproved by Albert and Linton~\cite{albert:growing-at-a-pe:} (and even more starkly by Vatter~\cite{vatter:permutation-cla}).  However, Klazar~\cite{klazar:overview-of-som} has suggested that their conjecture might be true for all finitely based classes.

Moving from general concerns to more local matters, throughout this work we have routinely required geometric griddability as a hypotheses, and it is natural to ask if this can be replaced by the weaker condition of strong rationality.  In general the answer is no, and essentially all attempts are thwarted by a particular strongly rational class.  We feel it might be edifying to ponder this class and what extensions of our results it \emph{does not} rule out, so we describe it in some detail.

To begin with we need the \emph{increasing oscillating antichain}.  To construct this antichain, take the set of increasing oscillations (oriented as on the far left of Figure~\ref{fig-osc}) of odd length at least three and `anchor' the two ends of the increasing oscillations by inflating the first and the greatest entry of each by the permutation $12$.  Thus the first element of the antichain is $231[12,12,1]=23451$, while the fourth element is
\[
241638597[12,1,1,1,1,1,1,12,1]=2\ 3\ 5\ 1\ 7\ 4\ 9\ 6\ 10\ 11\ 8,
\]
shown on the left of Figure~\ref{fig-u4} (which also gives a sketch of the proof that it \emph{is} an antichain; numerous formal proofs exist elsewhere).  Let $A$ denote this antichain.

\begin{figure}
\begin{center}
\begin{tabular}{ccc}
\psset{xunit=0.009in, yunit=0.009in}
\begin{pspicture}(0,0)(110,110)
\psaxes[dy=10,Dy=1,dx=10,Dx=1,tickstyle=bottom,showorigin=false,labels=none](0,0)(110,110)
\psframe[linecolor=darkgray,fillstyle=solid,fillcolor=lightgray,linewidth=0.02in](7,17)(23,33)
\psframe[linecolor=darkgray,fillstyle=solid,fillcolor=lightgray,linewidth=0.02in](87,97)(103,113)
\pscircle*(10,20){0.04in}
\pscircle*(20,30){0.04in}
\pscircle*(30,50){0.04in}
\pscircle*(40,10){0.04in}
\pscircle*(50,70){0.04in}
\pscircle*(60,40){0.04in}
\pscircle*(70,90){0.04in}
\pscircle*(80,60){0.04in}
\pscircle*(90,100){0.04in}
\pscircle*(100,110){0.04in}
\pscircle*(110,80){0.04in}
\end{pspicture}

&\rule{0.2in}{0pt}&

\psset{xunit=0.01in, yunit=0.01in}
\begin{pspicture}(5.8578644,0)(174.142136,110)
\pscircle*(15.8578644,69.1421356){0.04in}
\pscircle*(15.8578644,40.8578644){0.04in}
\pscircle*(30,55){0.04in}
\pscircle*(50,55){0.04in}
\pscircle*(70,55){0.04in}
\pscircle*(90,55){0.04in}
\pscircle*(110,55){0.04in}
\pscircle*(130,55){0.04in}
\pscircle*(150,55){0.04in}
\pscircle*(164.142136,69.1421356){0.04in}
\pscircle*(164.142136,40.8578644){0.04in}
\psline(15.8578644,69.1421356)(30,55)
\psline(15.8578644,40.8578644)(30,55)
\psline(30,55)(150,55)
\psline(150,55)(164.142136,69.1421356)
\psline(150,55)(164.142136,40.8578644)
\end{pspicture}
\\\\
\end{tabular}
\caption{On the left, a member of the infinite antichain $A$.  It is easiest to see that $A$ forms an antichain by considering the inversion graphs (or, permutation graphs) of its members (right), which form an infinite antichain of graphs under the induced subgraph order.}
\label{fig-u4}
\end{center}
\end{figure}
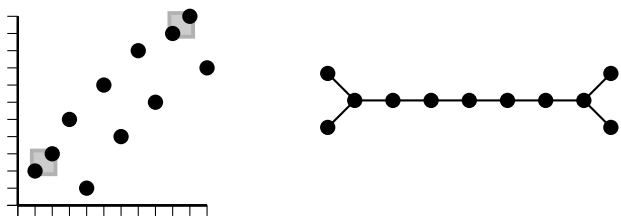

By definition we see that $A\subseteq\O[\{1,12\}]$.  Moreover, $\O$ can be seen to be strongly rational in several ways.  Perhaps the most systematic method is to consider the rank encodings of Albert, Atkinson, and Ru\v{s}kuc~\cite{albert:regular-closed-:}; in this encoding the class $\O$ and all its subclasses are in bijection with regular languages.

The class $\O[\{1,12\}]$ contains $A$ and thus is not pwo.  As observed in Proposition~\ref{prop-strong-rat-pwo}, this implies that $\O[\{1,12\}]$ is not strongly rational, and this fact dooms all naive generalisations of our results.  However, the rank encoding can be used to show that every \emph{finitely based} subclass of $\O[\{1,12\}]$ has a rational generating function, so one might optimistically hope that the following question has a positive answer.

\begin{question}
\label{ques-strong-rat-completion-enum}
If $\C$ and $\U$ are both strongly rational classes, does every finitely based subclass of $\langle\C\rangle$ (resp., $\C[\U]$) have an algebraic (resp., a rational) generating function?
\end{question}

As stated above, the conclusion about rationality holds for $\C=\O$ and $\U=\{1,12\}$.  It may be enlightening to answer this question in the special case where $\C=\O$ and $\U$ is an arbitrary strongly rational class.

Answering Question~\ref{ques-strong-rat-completion-enum} in general would almost surely require a greater understanding of the simple permutations in strongly rational classes.  Although Proposition~\ref{prop-strong-rat-indecomps} gives us a very good idea of the structure and enumeration of sum indecomposable permutations in a strongly rational class, its simple permutation analogue is still open:

\begin{conjecture}
If the class $\C$ is strongly rational, then the simple permutations in $\C$ have a rational generating function.
\end{conjecture}

\bigskip

\bibliographystyle{acm}
\bibliography{../refs}

\end{document}